\documentclass[12pt,reqno]{amsart}

\usepackage[usenames,dvipsnames]{color}
\usepackage{amssymb}
\usepackage{comment}
\usepackage{enumitem}
\usepackage{graphicx}
\usepackage{esint}
\usepackage{pdfpages}
\usepackage{svg}
\usepackage{caption}
\usepackage{mathrsfs}
\usepackage{xypic}
\usepackage[unicode,colorlinks=true,linkcolor=Red,citecolor=Green]{hyperref}
\AtBeginDocument{\addtocontents{toc}{\protect\setlength{\parskip}{0pt}}}

\SetLabelAlign{parright}{\parbox[t]{\labelwidth}{\raggedleft#1}}

\newtheorem{theorem}{Theorem}[section]
\newtheorem{proposition}[theorem]{Proposition}
\newtheorem{lemma}[theorem]{Lemma}
\newtheorem{corollary}[theorem]{Corollary}

\theoremstyle{definition}
\newtheorem{definition}[theorem]{Definition}

\theoremstyle{remark}
\newtheorem{remark}[theorem]{Remark}

\numberwithin{equation}{section}

\def\CIc{C^\infty_{\mathrm c}} % smooth compactly supported
\def\indic{\operatorname{1\hskip-2.75pt\relax l}} % Indicator function macro, courtesy of Colin Guillarmou

\newcommand{\AD}{\mathrm{AD}} % Ahlfors-David regular

\DeclareMathOperator{\const}{const}

\DeclareMathOperator{\FUP}{FUP}

\newcommand{\KT}{\mathrm{KT}} % Katz-Tao
\newcommand{\Lip}{\mathrm{Lip}} % Lipschitz

\renewcommand{\Re}{\operatorname{Re}}
\newcommand{\RKT}{\mathrm{RKT}} % rectangular Katz-Tao

\DeclareMathOperator{\supp}{supp}

\title{A $\tfrac{1}{64}$ spectral gap for surfaces with $\delta=\tfrac12$}

\author{Alex Cohen}
\email{alexcohen@nyu.edu}

\address{Courant Institute of Mathematical Sciences, New York University, 251 Mercer Street, New York, NY 10012}

\author{Semyon Dyatlov}
\email{dyatlov@math.mit.edu}

\address{Massachusetts Institute of Technology, 77 Massachusetts Ave, Cambridge, MA 02139}

\begin{document}

\begin{abstract}
We show that any convex co-compact hyperbolic surface with exponent of convergence of Poincar\'e series $\delta \in (\frac25,\frac{14}{27})$
has an essential spectral gap of size $\beta=\tfrac78(\tfrac12-\delta)+\tfrac{1}{32}\delta-\epsilon$
for any $\epsilon>0$.
In particular, for $\delta=\frac12$ this becomes $\beta=\tfrac{1}{64}-\epsilon$. 
We show existence of the gap by proving a new Fractal Uncertainty Principle,
using a two-ends Furstenberg theorem of O'Regan--Wu--Yi~\cite{ORegan-Wu-Yi}. 
\end{abstract}

\maketitle

%%%%%%%%%%%%%%%%%%%%%%%%%%%%%%%%%%%%%%%%%%%%%%%%%%%%%%%%%%%%%%%%%%%%%%%%%%%%%%%%
%%%%%%%%%%%%%%%%%%%%%%%%%%%%%%%%%%%%%%%%%%%%%%%%%%%%%%%%%%%%%%%%%%%%%%%%%%%%%%%%
\section{Introduction}

Let $M=\Gamma\backslash\mathbb H^2$ be a convex co-compact hyperbolic surface.
The Selberg zeta function is defined as follows:
\begin{equation}
  \label{e:zeta-def}
Z_M(s)=\prod_{\ell\in \mathcal L_M}\prod_{k=0}^\infty \big(1-e^{-(s+k)\ell}\big),\quad\Re s\gg 1,
\end{equation}
and it extends meromorphically to $s\in\mathbb C$. Here $\mathcal L_M$ is the set of
lengths of primitive closed geodesics on~$M$ (counted with multiplicity).
See for example the book of Borthwick~\cite[Chapters~10, 14, 15, 16]{Borthwick-Book} for more information. Our main result is
%%%%%%%%%%%%%%%%%%%%%%%%%%%%%%%%%%%%%%%%%%%%%%%%%%%%%%%%%%%%%%%%%%%%%%%%%%%%%%%%
\begin{theorem}
  \label{t:gap}
Let $\delta\in [0,1]$ be the
exponent of convergence of the Poincar\'e series of the group $\Gamma$. Then for each
\begin{equation}
  \label{e:gap}
\beta< \tfrac78(\tfrac12-\delta)+\tfrac{1}{32}\delta,
\end{equation}
the zeta function~\eqref{e:zeta-def} has only finitely many zeroes in $\{\Re s>\tfrac12-\beta\}$.
\end{theorem}
%%%%%%%%%%%%%%%%%%%%%%%%%%%%%%%%%%%%%%%%%%%%%%%%%%%%%%%%%%%%%%%%%%%%%%%%%%%%%%%%
In particular, if $\delta=\tfrac12$
then $M$ has an essential spectral gap of any size $\beta<\frac{1}{64}$.
In general, \eqref{e:gap} improves over the standard gap
$\beta=\max(0,\tfrac12-\delta)$ for $\delta\in (\frac25,\frac{14}{27})$.
See Figure~\ref{f:FUP-beta}. 

In general, we say that $M$ has an \emph{essential spectral gap}
of size $\beta\in\mathbb R$, if $Z_M(s)$ has only finitely many zeroes in $\{\Re s>\tfrac 12-\beta\}$.
Spectral gaps have applications to exponential local energy decay
for waves, local smoothing estimates, Strichartz estimates, length spectrum asymptotics,
and number theory; see for instance the surveys of Nonnenmacher~\cite{Nonnenmacher-OQC} and Zworski~\cite{Zworski-resonances}
for background on the spectral gap question 
and the papers of Dyatlov--Zahl~\cite{hgap} and Bourgain--Dyatlov~\cite{fullgap} for more
information on the setting used here. We list some of the previous results in~\S\ref{s:history} below.

%%%%%%%%%%%%%%%%%%%%%%%%%%%%%%%%%%%%%%%%%%%%%%%%%%%%%%%%%%%%%%%%%%%%%%%%%%%%%%%%
\begin{figure}
\includegraphics{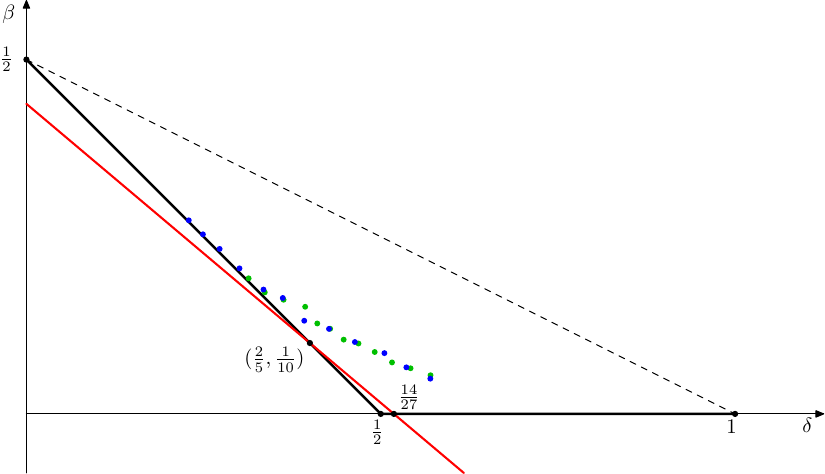}
\caption{The spectral gap of Theorem~\ref{t:gap}, marked in red.
The solid line is the standard gap $\beta=\max(0,\frac12-\delta)$.
The dashed line is the gap conjectured by Jakobson--Naud~\cite{Jakobson-Naud},
$\beta=\tfrac12-\tfrac12\delta$. The circles correspond to numerically computed
spectral gaps for symmetric 3-funneled surfaces (blue) and 4-funneled surfaces (green)
from the work of Borthwick--Weich~\cite[Figure~15]{Borthwick-Weich}, speficically $G^{I_1}_{100}$ in the notation of that paper.}
\label{f:FUP-beta}
\end{figure}
%%%%%%%%%%%%%%%%%%%%%%%%%%%%%%%%%%%%%%%%%%%%%%%%%%%%%%%%%%%%%%%%%%%%%%%%%%%%%%%%
Theorem~\ref{t:gap} relies on a fractal uncertainty principle, which follows the strategy introduced in~\cite{hgap}. To state it, define the following family of operators $\mathcal B_\chi(h)$ on~$L^2(\mathbb S^1)$
depending on the semiclassical parameter $h\in (0,1]$:
\begin{equation}
  \label{e:B-chi-def-1}
\mathcal B_\chi(h)f(y)=(2\pi h)^{-\tfrac12}\int_{\mathbb S^1}|y-y'|^{2i\over h}\chi(y,y')f(y')\,dy'.
\end{equation}
Here $|y-y'|$ denotes the Euclidean distance between $y,y'\in\mathbb S^1$, with $\mathbb S^1$
embedded as the unit circle into~$\mathbb R^2$, and the cutoff function $\chi$ satisfies
$$
\chi\in\CIc(\mathbb S^1_\Delta)\quad\text{where}\quad
\mathbb S^1_\Delta:=\{(y,y')\in\mathbb S^1\times\mathbb S^1\colon y\neq y'\}.
$$
For $\rho\in (0,1)$, denote by $\Lambda_\Gamma(h^\rho)\subset\mathbb S^1$ the $h^\rho$-neighborhood of the limit
set of $\Gamma$. For a set $X\subset\mathbb S^1$, denote by $\indic_X:L^2(\mathbb S^1)\to L^2(\mathbb S^1)$ the multiplication operator by the indicator function of $X$.

Now Theorem~\ref{t:gap} is obtained by combining either~\cite[Theorem~3]{hgap}
or the later result of Dyatlov--Zworski~\cite{DZ-TUG} with the following fractal uncertainty principle:
%%%%%%%%%%%%%%%%%%%%%%%%%%%%%%%%%%%%%%%%%%%%%%%%%%%%%%%%%%%%%%%%%%%%%%%%%%%%%%%%
\begin{theorem}
  \label{t:FUP}
Assume that $\beta$ satisfies~\eqref{e:gap}. Then there exists $\rho\in (0,1)$ depending only on~$\delta,\beta$
such that for all $\chi\in\CIc(\mathbb S^1_\Delta)$ and $h\in (0,1]$,
\begin{equation}
  \label{e:FUP}
\|\indic_{\Lambda_\Gamma(h^\rho)}\mathcal B_\chi(h)\indic_{\Lambda_\Gamma(h^\rho)}\|_{L^2(\mathbb S^1)\to L^2(\mathbb S^1)}\leq Ch^\beta,
\end{equation}
where the constant $C$ depends only on $M,\beta,\chi$, but not on~$h$.
\end{theorem}
%%%%%%%%%%%%%%%%%%%%%%%%%%%%%%%%%%%%%%%%%%%%%%%%%%%%%%%%%%%%%%%%%%%%%%%%%%%%%%%%

\subsection{General fractal uncertainty principle}

Fractal uncertainty principles hold for a more general class of semiclassical Fourier integral operators
and fractal sets, which we define now. Assume that $\mathcal U\subset \mathbb R^2$ is an open set and
\begin{equation}
  \label{e:FUP-Theta-chi}
\Theta\in C^\infty(\mathcal U;\mathbb R),\quad
\chi\in\CIc(\mathcal U).
\end{equation}
Consider the operator $\mathcal B(h):L^2(\mathbb R)\to L^2(\mathbb R)$ of the form
\begin{equation}
  \label{e:B-defined-general}
\mathcal B(h)f(x)=(2\pi h)^{-\tfrac12}\int_{\mathbb R} e^{{i\over h}\Theta(x,y)}
\chi(x,y)f(y)\,dy.
\end{equation}
We assume that the phase function $\Theta$ satisfies the nondegeneracy condition
\begin{equation}
  \label{e:phase-nondegenerate}
\partial_x\partial_y\Theta\neq 0\quad\text{on }\mathcal U.
\end{equation}
The operator $\mathcal B_\chi(h)$ from~\eqref{e:B-chi-def-1} has the form~\eqref{e:B-defined-general},
if we cut the circle $\mathbb S^1$ to identify it with an interval (which can be achieved by a partition of unity on $\chi$) and put, with $\theta,\theta'\in\mathbb R$ the angle coordinates on the circle,
\begin{equation}
  \label{e:our-Theta-def}
\begin{aligned}
\mathcal U&=\{(\theta,\theta')\in\mathbb R\colon \theta-\theta'\notin 2\pi\mathbb Z\},
\\
\Theta(\theta,\theta')&=2\log|e^{i\theta}-e^{i\theta'}|=2\log\bigg|2\sin\Big({\theta-\theta'\over 2}\Big)\bigg|.
\end{aligned}
\end{equation}
Another operator of the form~\eqref{e:B-defined-general} is the cutoff Fourier transform,
corresponding to $\Theta(x,y)=-xy$.

Our proof exploits a key new condition which we call the \emph{nonlinear phase condition}:
\begin{equation}
  \label{e:phase-nonlinear}
\partial_x\partial_y\log|\partial_x\partial_y\Theta|\neq 0\quad\text{on }\mathcal U.
\end{equation}
This condition fails for $\Theta(x,y)=-xy$ but it holds for the phase defined in~\eqref{e:our-Theta-def}:
$$
\begin{aligned}
\partial_\theta\partial_{\theta'} \Theta &= \frac12 \sin^{-2}\Big({\theta-\theta'\over 2}\Big),
\\
\partial_\theta\partial_{\theta'} \log |\partial_\theta\partial_{\theta'}\Theta| &=
-\frac12 \sin^{-2}\Big({\theta-\theta'\over 2}\Big).
\end{aligned}
$$
The key property of the sets $\Lambda_\Gamma$ is \emph{Ahlfors--David regularity}.
We follow~\cite[Definition~1.1]{fullgap}:
%%%%%%%%%%%%%%%%%%%%%%%%%%%%%%%%%%%%%%%%%%%%%%%%%%%%%%%%%%%%%%%%%%%%%%%%%%%%%%%%
\begin{definition}
\label{d:ADR}
Let $X\subset\mathbb R$ be a nonempty closed set and $\delta\in [0,1]$, $C_{\AD}\geq 1$,
$0\leq r_1\leq r_0\leq \infty$. We say that $X$ is \emph{$\delta$-dimensional AD regular on scales $r_1$ to $r_0$ with constant $C_{\AD}$} if there exists a Borel measure $\mu_X$ on~$\mathbb R$ such that:
\begin{itemize}
\item $\mu_X$ is supported on~$X$, that is $\mu_X(\mathbb R\setminus X)=0$;
\item for each interval $I$ of length $|I|\in [r_1,r_0]$, we have $\mu_X(I)\leq C_{\AD}|I|^\delta$;
\item if additionally $I$ is centered at a point in~$X$, then $\mu_X(I)\geq C_{\AD}^{-1}|I|^\delta$.
\end{itemize}
\end{definition}
%%%%%%%%%%%%%%%%%%%%%%%%%%%%%%%%%%%%%%%%%%%%%%%%%%%%%%%%%%%%%%%%%%%%%%%%%%%%%%%%
If $M=\Gamma\backslash\mathbb H^2$ is a convex co-compact hyperbolic surface,
then the limit set $\Lambda_\Gamma$ is $\delta$-dimensional AD regular on scales~0 to~1
with some constant $C_{\AD}$, see for example~\cite[Lemma~14.13 and Theorem~14.14]{Borthwick-Book}.
The neighborhood $\Lambda_\Gamma(h^\rho)$ is then $\delta$-dimensional AD regular on scales $h^\rho$ to~1
with constant $8C_{\AD}$, as follows from Lemma~\ref{l:AD-nbhd} below. Thus Theorem~\ref{t:FUP} follows from the following more general fractal uncertainty principle:
%%%%%%%%%%%%%%%%%%%%%%%%%%%%%%%%%%%%%%%%%%%%%%%%%%%%%%%%%%%%%%%%%%%%%%%%%%%%%%%%
\begin{theorem}
  \label{t:FUP-intro-general}
Assume that
\begin{equation}
  \label{e:delta-beta-rho}
\delta\in [\tfrac25,\tfrac{14}{27}],\quad
\epsilon\in (0,0.1],\quad
\beta:=\tfrac78(\tfrac12-\delta)+\tfrac{1}{32}\delta-2\epsilon,\quad
\rho:=1-\epsilon.
\end{equation}
Let $\mathcal B(h)$ be an operator of the form~\eqref{e:B-defined-general},
satisfying the assumptions~\eqref{e:FUP-Theta-chi}, \eqref{e:phase-nondegenerate},
and~\eqref{e:phase-nonlinear}.
Assume that $h\in (0,1]$ and $X,Y\subset\mathbb R$ are $\delta$-dimensional AD regular sets
on scales $h^\rho$ to~1 with constant $C_{\AD}$. Then we have
\begin{equation}
  \label{e:FUP-intro-general}
\|\indic_X \mathcal B(h)\indic_Y\|_{L^2(\mathbb R)\to L^2(\mathbb R)}
\leq C h^{\beta}
\end{equation}
where the constant $C$ depends only on $\delta,\epsilon,C_{\AD},\Theta,\chi$,
but not on~$X,Y,h$.
\end{theorem}
%%%%%%%%%%%%%%%%%%%%%%%%%%%%%%%%%%%%%%%%%%%%%%%%%%%%%%%%%%%%%%%%%%%%%%%%%%%%%%%%
In~\S\ref{s:reductions}, we reduce Theorem~\ref{t:FUP-intro-general}
to an additive energy bound, Proposition~\ref{l:additive-bound};
this bears some similarities with the reduction in~\cite[\S5]{hgap} but the additive
quantity used is different.
Proposition~\ref{l:additive-bound} is proved in~\S\ref{s:additive-bound}. The main ingredient
is the recent two-ends Furstenberg inequality of O'Regan--Wu--Yi~\cite{ORegan-Wu-Yi},
which in turn builds on the work of Demeter--Wang~\cite{Demeter-Wang}
and Wang--Wu~\cite{Wang-Wu-1,Wang-Wu-2}.

We emphasize that Theorem~\ref{t:FUP-intro-general} cannot hold for the Fourier
phase $\Theta(x,y)=-xy$. Indeed, in this setting Dyatlov--Jin~\cite[Proposition~3.17]{oqm}
give for any $\delta\in [0,\tfrac12]$ an example of a sequence of $\delta$-dimensional AD regular sets $X_n,Y_n$ for which the norm~\eqref{e:FUP-intro-general} is $\gtrsim h^{\beta_n}$
for some $\beta_n\to \tfrac 12-\delta$. In particular, for $\delta=\frac12$ 
we can have $\beta_n\sim (\log C_{\AD,n})^{-1}$. 
Kangabire~\cite[Theorem~1]{Kangabire-FUP} gives an example with similar properties
when $\delta\in [\frac12,1]$. That is, the fractal uncertainty exponent in the case
of the Fourier transform has to depend on more than just the dimensions of the sets.

%%%%%%%%%%%%%%%%%%%%%%%%%%%%%%%%%%%%%%%%%%%%%%%%%%%%%%%%%%%%%%%%%%%%%%%%%%%%%%%%
\subsection{Brief overview of history}
\label{s:history}

We provide a very brief overview of the history of the spectral gap question.
Two classical spectral gaps are the \emph{Lax--Phillips} spectral gap $\beta=0$
and the \emph{Patterson--Sullivan} spectral gap $\beta=\tfrac12-\delta$~\cite{Patterson,Sullivan}.
There are various improvements over these gaps, due to
Naud~\cite{Naud-Gap}, Dyatlov--Zahl~\cite{hgap},
Bourgain--Dyatlov~\cite{fullgap,hyperfup}, Dyatlov--Jin~\cite{regfup}, and Jin--Zhang~\cite{Jin-Zhang},
 which we summarize in the table below.
Below `Range of $\delta$' refers to the range of the values of $\delta$ for which the improvement
over the standard gap, namely $\beta-\max(0,\tfrac12-\delta)$, is positive,
and `Improvement' refers to the size of that improvement.
Many results that apply to $\delta=\tfrac12$ also apply to nearby values of~$\delta$,
but we choose to replace the ranges by just $\tfrac12$ to simplify the table below.
Below $\mathbf k$ refers to an unspecified small positive universal constant
and $\mathbf f(\bullet)$ refers to an unspecified function of the objects in the parentheses. 
\smallskip
\def\xstrut{\vrule width0pt height13pt depth6pt}
\begin{center}
\begin{tabular}{|l|l|l|}
\hline
\xstrut Paper & Range of $\delta$ & Improvement $\beta-\max(0,\tfrac12-\delta)$ \\ 
\hline\hline
\xstrut \cite{Naud-Gap} &  $(0,\tfrac12]$ & Unspecified \\
\hline
\xstrut \cite{hgap} & $\tfrac12$ & $\exp(-\mathbf k^{-1}(1+\log^{14}C_{\AD}))$ \\
\hline
\xstrut \cite{fullgap} & $[\tfrac12,1)$ & $\mathbf f(\delta,C_{\AD})$ \\
\hline
\xstrut \cite{Jin-Zhang} & $[\tfrac12,1)$ & $\exp\big[ -\exp\big( \mathbf k^{-1} (C_{\AD}(1-\delta)^{-1})^{\mathbf k^{-1}(1-\delta)^{-2}}\big)\big]$ \\
\hline
\xstrut \cite{regfup} & $(0,\tfrac12]$ & $(5C_{\AD})^{-{160\over \delta(1-\delta)}}$ \\
\hline
\xstrut \cite{hyperfup} & $(0,\tfrac12]$ & $\mathbf f(\delta)$ \\
\hline
\xstrut Theorem~\ref{t:gap} & $\tfrac12$ & $\tfrac{1}{64}-$ \\
\hline
\end{tabular}
\end{center}
\medskip
All the results above except~\cite{Naud-Gap} prove a fractal uncertainty principle.
All of these except Theorem~\ref{t:gap} work also for the Fourier phase $\Theta(x,y)=-xy$.
The only universal bounds above (that is, $\beta$ depending only on $\delta$)
are Theorem~\ref{t:gap} and the result of~\cite{hyperfup}. The latter uses
a Fourier decay property for the limit sets $\Lambda_\Gamma$ which relies
on their structure beyond AD regularity (more precisely, on the nonlinearity of the transformations
in the group $\Gamma$).
We mention here also the recent paper by Lequen--Sahlsten~\cite{Lequen-Sahlsten}
on Fourier decay for Schottky limit sets.

We finally mention the conjecture of Jakobson--Naud~\cite{Jakobson-Naud}
that every convex co-compact hyperbolic surface has an essential spectral
gap of size $\beta=\tfrac12-\tfrac12\delta$. We refer to~\cite[\S16.3.2]{Borthwick-Book}
for numerical evidence regarding essential spectral gaps.

\noindent\textbf{Acknowledgements.}
The idea of using projection theory together with the nonlinear phase condition~\eqref{e:phase-nonlinear}
to improve the known fractal uncertainty principle, as well as the content of \S\S\ref{s:basics}--\ref{s:reductions}, is entirely due to the authors. Part of these sections is copied
from the work in progress~\cite{ehalf}; if the present paper is submitted for publication
after~\cite{ehalf} is released, then this part will be replaced by references.
For~\S\ref{s:additive-bound} we used OpenAI Codex to find a result in projection theory
which works easiest to get the needed additive energy bound. Given the statement
of Proposition~\ref{l:additive-bound} and the reference~\cite{ORegan-Wu-Yi},
it found~\cite[Theorem~4.4]{ORegan-Wu-Yi}
and gave a first detailed sketch of the proof of Proposition~\ref{l:additive-bound} using it.
The final version of this paper is entirely human written and human verified.
Alex Cohen was supported by a Clay Research Fellowship. Semyon Dyatlov was supported by
the NSF grant DMS-2400090, the Simons Fellowship SFI-MPS-SFM-00005584,
and has enjoyed the hospitality of the Mathematical Sciences Research Institute in Berkeley, California, during Fall 2024 (supported by the NSF grant DMS-1928930).

%%%%%%%%%%%%%%%%%%%%%%%%%%%%%%%%%%%%%%%%%%%%%%%%%%%%%%%%%%%%%%%%%%%%%%%%%%%%%%%%
%%%%%%%%%%%%%%%%%%%%%%%%%%%%%%%%%%%%%%%%%%%%%%%%%%%%%%%%%%%%%%%%%%%%%%%%%%%%%%%%
\section{Preliminaries}
\label{s:basics}

%%%%%%%%%%%%%%%%%%%%%%%%%%%%%%%%%%%%%%%%%%%%%%%%%%%%%%%%%%%%%%%%%%%%%%%%%%%%%%%%
\subsection{Basic concepts}

We first introduce some basic notation:
\begin{itemize}
\item For a finite set $X$, denote by $\#(X)$ its cardinality.
\item We use the sum set notation for $X,Y\subset\mathbb R^d$:
$$
X+Y:=\{x+y\mid x\in X,\ y\in Y\}.
$$
\item For $r>0$, an \emph{$r$-cube} in $\mathbb R^d$ is a translate of $[0,r]^d$.
If $d=1$, we call an $r$-cube an \emph{$r$-interval}. If $d=2$, it is often called an \emph{$r$-square}.
\item All the subsets of $\mathbb R^d$ we use will be Borel measurable.
\item For $X\subset\mathbb R^d$, denote by $|X|$ its Lebesgue measure.
\item For $X\subset\mathbb R^d$ and $r>0$, define the quantities
\begin{equation}
  \label{e:M-r-X-def}
|X|_r \in [0,\infty],\quad
\mathbf M_r(X)\in [0,r^d]
\end{equation}
as follows: $|X|_r$ is the minimal number of $r$-cubes needed to cover $X$,
and $\mathbf M_r(X)$ is the supremum of $|X\cap Q|$ over all $r$-cubes $Q\subset\mathbb R^d$.
Define also the $r$-neighborhood
\begin{equation}
  \label{e:X-r-def}
X(r):=X+[-r,r]^d.
\end{equation}
\item For $r>0$, define the grid of $r$-intervals as
$$
\mathcal D_r(\mathbb R):=\big\{[jr,(j+1)r]\colon j\in\mathbb Z\big\}.
$$
Using these, we define the grid of $r$-cubes in $\mathbb R^d$ as
\begin{equation}
  \label{e:D-r-def}
\mathcal D_r(\mathbb R^d):=\{\mathbf x_1\times\cdots\times \mathbf x_d\colon \mathbf x_1,\dots,\mathbf x_d\in\mathcal D_r(\mathbb R)\}.
\end{equation}
Note that the cubes in $\mathcal D_r(\mathbb R^d)$ are nonoverlapping (i.e. the Lebesgue measure of the
intersection of any two distinct cubes is 0) and
$$
\mathbb R^d = \bigcup_{\mathbf p\in\mathcal D_r(\mathbb R^d)}\mathbf p.
$$
For $X\subset\mathbb R^d$, denote by $\mathcal D_r(X)$ the set of grid cubes intersecting $X$:
\begin{equation}
  \label{e:D-r-X-def}
\mathcal D_r(X):=\{\mathbf p\in\mathcal D_r(\mathbb R^d)\colon \mathbf p\cap X\neq\emptyset\}.
\end{equation}
\item We say $r$ is a \emph{dyadic} number if $r=2^{n}$ for some $n\in\mathbb Z$.
%\item If $X\subset\mathbb R^d$ and $r>0$, we say that $X$ is \emph{$r$-separated}
%if for each $x,y\in X$ with $x\neq y$ we have $|x-y|>r$.
\end{itemize}
%%%%%%%%%%%%%%%%%%%%%%%%%%%%%%%%%%%%%%%%%%%%%%%%%%%%%%%%%%%%%%%%%%%%%%%%%%%%%%%%
We record here some basic properties of the above objects:
\begin{itemize}
\item If $X\subset Y\subset\mathbb R^d$ and $r>0$ then
\begin{equation}
  \label{e:monotone-set}
|X|_r\leq |Y|_r,\quad
\mathbf M_r(X)\leq \mathbf M_r(Y).
\end{equation}
\item If $0<r_1\leq r_2$ and $X\subset \mathbb R^d$ then
\begin{align}
  \label{e:r-cover-monotone}
|X|_{r_2}
&\leq |X|_{r_1}\leq \Big\lceil{r_2\over r_1}\Big\rceil^d |X|_{r_2},
\\
  \label{e:M-r-monotone}
\mathbf M_{r_1}(X)
&\leq \mathbf M_{r_2}(X)\leq \Big\lceil{r_2\over r_1}\Big\rceil^d\mathbf M_{r_1}(X).
\end{align}
Indeed, each $r_1$-cube is contained in some $r_2$-cube, and each $r_2$-cube can be covered by $\lceil r_2/r_1\rceil^d$ many $r_1$-cubes.
\item For each $X\subset\mathbb R^d$ and $r_1,r_2>0$ we have
\begin{equation}
  \label{e:X-r-nbhd}
|X(r_2)|_{r_1+2r_2}\leq |X|_{r_1}.
\end{equation}
Indeed, the $r_2$-neighborhood of an $r_1$-cube is an $r_1+2r_2$-cube.
\item For each $X\subset\mathbb R^d$ and $r>0$ we have
\begin{align}
  \label{e:volume-covering-M}
|X|&\leq |X|_r\,\mathbf M_r(X),
\\
  \label{e:volume-nbhd}
 (r/2)^d|X|_r&\leq |X(r)|\leq (3r)^d|X|_r,
\\
  \label{e:D-r-bound}
|X|_r&\leq \#(\mathcal D_r(X))\leq 3^d|X|_r
\end{align}
To see the first bound in~\eqref{e:volume-nbhd}, let
$x_1,\dots,x_L\in X$ be a maximal set such that
the $r/2$-cubes centered at $x_j$ do not intersect. These
cubes are contained in~$X(r)$, so comparing the volume we see that
$L(r/2)^d\leq |X(r)|$.
On the other hand, $X$ is covered by the $r$-cubes centered at $x_j$,
so $|X|_r\leq L$.
 
To see the second bound in~\eqref{e:volume-nbhd}, note that
$X$ can be covered by $|X|_r$ many $r$-cubes, so $X(r)$
can be covered by $|X|_r$ many $3r$-cubes.

To see the bound~\eqref{e:D-r-bound}, note that
$$
X\ \subset\ \bigcup_{\mathbf p\in\mathcal D_r(X)}\mathbf p\ \subset\ X(r).
$$
The first containment gives the first bound in~\eqref{e:D-r-bound}
and the second containment together with~\eqref{e:volume-nbhd} gives the second bound.
%\item 
%If $N\geq 3$ is an integer and $X\subset\mathbb R^d$ is $r\sqrt{d}$-separated,
%then we can write
%\begin{equation}
%  \label{e:separated-decomposition}
%X=\bigcup_{\vec n\in \{0,\dots,N-1\}^d} X_{\vec n}
%\end{equation}
%where each $X_{\vec n}$ is $(N-2)r$-separated. To see this,
%define for each $\vec m\in \mathbb Z^d$ and $\vec n\in \{0,\dots,N-1\}^d$ 
%the grid $r$-cube
%$$
%\mathbf p_{\vec m,\vec n}:=r(N\vec m+\vec n)+[0,r]^d,\quad
%\mathbb R^d=\bigcup_{\vec m\in\mathbb Z^d\atop\vec n\in \{0,\dots,N-1\}^d}
%\mathbf p_{\vec m,\vec n}
%$$
%and put
%$$
%X_{\vec n}:=X\cap \bigcup_{\vec m\in\mathbb Z^d} \mathbf p_{\vec m,\vec n}.
%$$
%Then~\eqref{e:separated-decomposition} holds. Each $\mathbf p_{\vec m,\vec n}$
%has diameter $r\sqrt{d}$ and thus contains at most 1 point in~$X$.
%For two different values $\vec m,\vec m'$, the cubes $\mathbf p_{\vec m,\vec n}$
%and $\mathbf p_{\vec m',\vec n}$ are more than $(N-2)r$ far away from each other. Thus each set $X_{\vec n}$
%is $(N-2)r$-separated.
\end{itemize}

%%%%%%%%%%%%%%%%%%%%%%%%%%%%%%%%%%%%%%%%%%%%%%%%%%%%%%%%%%%%%%%%%%%%%%%%%%%%%%%%
\subsection{AD regular sets}

We now give some basic properties of AD regular sets, mostly from~\cite{fullgap}.
We start with a covering bound
which is a consequence of~\cite[Lemma~2.8]{fullgap}:
%%%%%%%%%%%%%%%%%%%%%%%%%%%%%%%%%%%%%%%%%%%%%%%%%%%%%%%%%%%%%%%%%%%%%%%%%%%%%%%%
\begin{lemma}
  \label{l:AD-cover}
Let $X\subset\mathbb R$ be a $\delta$-dimensional AD regular set on scales $r_1$ to $r_0$ with constant $C_{\AD}$, $I$ be an interval, $r>0$, and assume that $r_1\leq r\leq |I|\leq r_0$. Then
\begin{equation}
  \label{e:AD-cover}
|X\cap I|_r\leq 12C_{\AD}^2 \Big({|I|\over r}\Big)^\delta.
\end{equation}
\end{lemma}
%%%%%%%%%%%%%%%%%%%%%%%%%%%%%%%%%%%%%%%%%%%%%%%%%%%%%%%%%%%%%%%%%%%%%%%%%%%%%%%%
A further consequence of this is
%%%%%%%%%%%%%%%%%%%%%%%%%%%%%%%%%%%%%%%%%%%%%%%%%%%%%%%%%%%%%%%%%%%%%%%%%%%%%%%%
\begin{lemma}
  \label{l:AD-cover-2}
Let $X\subset \mathbb R$ be a $\delta$-dimensional AD regular set on scales $r\in (0,1]$ to~1
with constant $C_{\AD}$ and $I$ be an interval such that $|I|\leq 100$. Then
\begin{equation}
  \label{e:AD-cover-2}
|X\cap I|_r\leq 1200 C_{\AD}^2\Big({|I|\over r}+1\Big)^\delta.
\end{equation}
\end{lemma}
%%%%%%%%%%%%%%%%%%%%%%%%%%%%%%%%%%%%%%%%%%%%%%%%%%%%%%%%%%%%%%%%%%%%%%%%%%%%%%%%
\begin{proof}
We consider three cases:
\begin{enumerate}
\item $|I|<r$. Then $|X\cap I|_r\leq 1$ and~\eqref{e:AD-cover-2} follows.
\item $r\leq |I|\leq 1$. Then~\eqref{e:AD-cover-2} follows from Lemma~\ref{l:AD-cover}.
\item $1< |I|\leq 100$. Then we can cover $I$ by 100 intervals $I_1,\dots,I_{100}$
of length~1 each and estimate $|X\cap I_k|_r$ by Lemma~\ref{l:AD-cover}:
$$
|X\cap I|_r\leq \sum_{k=1}^{100}|X\cap I_k|_r
\leq 1200C_{\AD}^2\Big({1\over r}\Big)^\delta
$$
which gives~\eqref{e:AD-cover-2}.
\end{enumerate}
\end{proof}
%%%%%%%%%%%%%%%%%%%%%%%%%%%%%%%%%%%%%%%%%%%%%%%%%%%%%%%%%%%%%%%%%%%%%%%%%%%%%%%%
The next lemma is~\cite[Lemma~2.3]{fullgap}, which gives AD regularity of neighborhoods
of AD regular sets:
%%%%%%%%%%%%%%%%%%%%%%%%%%%%%%%%%%%%%%%%%%%%%%%%%%%%%%%%%%%%%%%%%%%%%%%%%%%%%%%%
\begin{lemma}
  \label{l:AD-nbhd}
Assume that $X$ is a $\delta$-dimensional AD regular set on scales $r_1$ to $r_0\geq 2r_1$
with constant $C_{\AD}$, and fix $T\geq 1$. Then the neighborhood
$X(Tr_1)$ is $\delta$-dimensional AD regular on scales $2r_1$ to $r_0$ with constant
$4TC_{\AD}$.
\end{lemma}
%%%%%%%%%%%%%%%%%%%%%%%%%%%%%%%%%%%%%%%%%%%%%%%%%%%%%%%%%%%%%%%%%%%%%%%%%%%%%%%%
We now give~\cite[Lemma~2.4]{fullgap} which shows that AD regularity is preserved under bi-Lipschitz maps.
%%%%%%%%%%%%%%%%%%%%%%%%%%%%%%%%%%%%%%%%%%%%%%%%%%%%%%%%%%%%%%%%%%%%%%%%%%%%%%%%
\begin{lemma}
  \label{l:AD-diffeo}
Assume that $F:\mathbb R\to\mathbb R$ is a $C^1$ diffeomorphism such that
$C_F^{-1}\leq |\partial_x F|\leq C_F$ for some constant $C_F\geq 1$. Let
$X$ be a $\delta$-dimensional AD regular set on scales $r_1$ to $r_0\geq C_F^2 r_1$
with constant $C_{\AD}$.
Then $F(X)$ is a $\delta$-dimensional AD regular set on scales $C_Fr_1$ to $C_F^{-1}r_0$
with constant $C_FC_{\AD}$.
\end{lemma}
%%%%%%%%%%%%%%%%%%%%%%%%%%%%%%%%%%%%%%%%%%%%%%%%%%%%%%%%%%%%%%%%%%%%%%%%%%%%%%%%

The next lemma is~\cite[Lemma~2.2]{fullgap} which allows to increase the upper scale of AD regularity:
%%%%%%%%%%%%%%%%%%%%%%%%%%%%%%%%%%%%%%%%%%%%%%%%%%%%%%%%%%%%%%%%%%%%%%%%%%%%%%%%
\begin{lemma}
  \label{l:AD-upper}
Assume that $X$ is a $\delta$-dimensional AD regular set on scales $r_1$
to $r_0$ with constant $C_{\AD}$. Then for each $T\geq 1$,
$X$ is $\delta$-dimensional AD regular on scales $r_1$ to $Tr_0$
with constant $2TC_{\AD}$. 
\end{lemma}
%%%%%%%%%%%%%%%%%%%%%%%%%%%%%%%%%%%%%%%%%%%%%%%%%%%%%%%%%%%%%%%%%%%%%%%%%%%%%%%%

Finally, we give~\cite[Lemma~2.7]{fullgap}, which shows how
to split an AD regular set into a union of smaller AD regular sets.
We specialize it to the case when $\delta\leq \frac23$.
%%%%%%%%%%%%%%%%%%%%%%%%%%%%%%%%%%%%%%%%%%%%%%%%%%%%%%%%%%%%%%%%%%%%%%%%%%%%%%%%
\begin{lemma}
  \label{l:AD-split}
Assume that $X$ is a $\delta$-dimensional AD regular set on scales $r_1$ to $r_0$ with constant
$C_{\AD}$, and assume that $\delta\leq \frac23$ and $(4C_{\AD})^6 r_1\leq r'\leq r_0$. Then there exists a collection of disjoint intervals $\mathcal J$ such that
$$
X=\bigsqcup_{J\in\mathcal J}(X\cap J);\quad
(4C_{\AD})^{-6}r'\leq |J|\leq r'\quad\text{for all }J\in\mathcal J,
$$
and each $X\cap J$ is $\delta$-dimensional AD regular on scales $r_1$ to $r'$ with constant
$4096 C_{\AD}^7$.
\end{lemma}
%%%%%%%%%%%%%%%%%%%%%%%%%%%%%%%%%%%%%%%%%%%%%%%%%%%%%%%%%%%%%%%%%%%%%%%%%%%%%%%%

%%%%%%%%%%%%%%%%%%%%%%%%%%%%%%%%%%%%%%%%%%%%%%%%%%%%%%%%%%%%%%%%%%%%%%%%%%%%%%%%
\subsection{Simple and moderate regions in the plane}
\label{s:moderate}

Here we study regions in the plane which are nice enough when intersected with the product
of two $\delta$-dimensional sets. This will be used in the proof of rectangular Katz--Tao
property in~\S\ref{s:RKT-proof} below.

We first give a definition inspired by the concept of parallel and transversal lines:
%%%%%%%%%%%%%%%%%%%%%%%%%%%%%%%%%%%%%%%%%%%%%%%%%%%%%%%%%%%%%%%%%%%%%%%%%%%%%%%%
\begin{definition}
\label{d:partran}
Let $J\subset\mathbb R$ be an interval and $f,g:J\to \mathbb R$ be continuous.
\begin{itemize}
\item We say that $f,g$ are \emph{parallel}, if
\begin{equation}
  \label{e:moderate-region}
0<|g(y)-f(y)|\leq 10|g(y')-f(y')|\quad\text{for all }y,y'\in J.
\end{equation}
\item We say that $f,g$ are \emph{transversal}, if for each $c\in\mathbb R$,
the set $\{y\in J\colon g(y)-f(y)=c\}$ has at most 10 points.
\end{itemize}
\end{definition}
%%%%%%%%%%%%%%%%%%%%%%%%%%%%%%%%%%%%%%%%%%%%%%%%%%%%%%%%%%%%%%%%%%%%%%%%%%%%%%%%
One way to check that two functions are transversal is by using their derivatives:
%%%%%%%%%%%%%%%%%%%%%%%%%%%%%%%%%%%%%%%%%%%%%%%%%%%%%%%%%%%%%%%%%%%%%%%%%%%%%%%%
\begin{lemma}
  \label{l:transversal-der}
Let $J\subset\mathbb R$ be an interval and assume that $f,g\in C^\infty(J;\mathbb R)$ satisfy
at least one of the following conditions:
\begin{itemize}
\item $f'(y)\neq g'(y)$ for all $y\in J$; or
\item $f''(y)\neq g''(y)$ for all $y\in J$.
\end{itemize}
Then $f,g$ are transversal.
\end{lemma}
 %%%%%%%%%%%%%%%%%%%%%%%%%%%%%%%%%%%%%%%%%%%%%%%%%%%%%%%%%%%%%%%%%%%%%%%%%%%%%%%%
\begin{proof}
By Rolle's Theorem, if the set $\{y\in J\colon g(y)-f(y)=c\}$ contains at least 2 points,
then $g'-f'$ has a zero on $J$. By the same theorem applied twice,
if the set above contains at least 3 points, then $g''-f''$ has a zero on~$J$.
\end{proof}
%%%%%%%%%%%%%%%%%%%%%%%%%%%%%%%%%%%%%%%%%%%%%%%%%%%%%%%%%%%%%%%%%%%%%%%%%%%%%%%%
A \emph{vertically simple region} is a subset of $\mathbb R^2$ of the form
\begin{equation}
  \label{e:v-simple-region}
\Omega=\{(y_1,y_2)\colon y_1\in J,\ f(y_1)\leq y_2\leq g(y_1)\}
\end{equation}
where $J\subset\mathbb R$ is an interval and $f,g:J\to\mathbb R$ are two functions
such that $f\leq g$ on~$J$. Alternatively one can consider a \emph{horizontally simple region}
by switching the roles of~$y_1$ and $y_2$; the results presented here apply to such regions as well.

We will study simple regions where the functions $f,g$ lie in the set
$$
\Lip_{10}(J):=\{f:J\to [-10, 10]\text{ such that for all }y,y'\in J, |f(y)-f(y')|\leq 10|y-y'|\}.
$$
The following definition imposes the additional condition that
the region have roughly constant thickness:
%%%%%%%%%%%%%%%%%%%%%%%%%%%%%%%%%%%%%%%%%%%%%%%%%%%%%%%%%%%%%%%%%%%%%%%%%%%%%%%%
\begin{definition}
  \label{d:moderate-region}
Let $\Omega\subset\mathbb R^2$. We say $\Omega$ is a \emph{vertically moderate region},
if it has the form~\eqref{e:v-simple-region}
where $J\subset [-10, 10]$, $f,g\in \Lip_{10}(J)$, $f<g$ on $J$, and $f,g$ are parallel.
\end{definition}
%%%%%%%%%%%%%%%%%%%%%%%%%%%%%%%%%%%%%%%%%%%%%%%%%%%%%%%%%%%%%%%%%%%%%%%%%%%%%%%%
Vertically moderate regions are useful to us because of the following
%%%%%%%%%%%%%%%%%%%%%%%%%%%%%%%%%%%%%%%%%%%%%%%%%%%%%%%%%%%%%%%%%%%%%%%%%%%%%%%%
\begin{lemma}
  \label{l:moderate-regular}
Let $\Omega\subset\mathbb R^2$ be a vertically moderate region,
$r\in (0,1]$, and $Y\subset \mathbb R$ be a $\delta$-dimensional AD regular set
on scales $r$ to~1 with constant $C_{\AD}$. Then
\begin{equation}
  \label{e:moderate-regular}
|\Omega\cap Y^2|_r\leq 3\cdot 10^8 C_{\AD}^4|\Omega|_r^\delta.
\end{equation}
\end{lemma}
%%%%%%%%%%%%%%%%%%%%%%%%%%%%%%%%%%%%%%%%%%%%%%%%%%%%%%%%%%%%%%%%%%%%%%%%%%%%%%%%
\begin{proof}
1. We write~$\Omega$ in the form~\eqref{e:v-simple-region} and define the maximal
thickness of~$\Omega$ as
$$
T_\Omega:=\max_{y\in J} (g(y)-f(y)).
$$
We first estimate $|\Omega|_r$ from below in terms of the width $|J|$ and the thickness $T_\Omega$.
Take a $2r$-separated set
$$
E\subset J,\quad
\#(E)=\Big\lfloor {|J|\over 2r}\Big\rfloor+1.
$$
For each $y_1\in E$, take a $2r$-separated set
$$
S_{y_1}\subset [f(y_1),g(y_1)],\quad
\#(S_{y_1})=\Big\lfloor {g(y_1)-f(y_1)\over 2r}\Big\rfloor+1
\geq \Big\lfloor{T_\Omega\over 20r} \Big\rfloor +1
$$
where in the last inequality we use the parallel condition~\eqref{e:moderate-region}.

The set
$$
\{(y_1,y_2)\colon y_1\in E,\ y_2\in S_{y_1}\}
$$
is contained in~$\Omega$ and no two distinct points in this set can lie in the same $r$-square.
Thus
\begin{equation}
  \label{e:mod-reg-int-1}
\begin{aligned}
|\Omega|_r &\geq
\bigg(\Big\lfloor {|J|\over 2r}\Big\rfloor+1\bigg)\bigg(\Big\lfloor{T_\Omega\over 20r} \Big\rfloor +1\bigg)
\\
&\geq \frac14\bigg({|J|\over 2r}+1\bigg)\bigg({T_\Omega\over 20r} +1\bigg).
\end{aligned}
\end{equation}

\noindent 2. We now cover $\Omega\cap Y^2$ by $r$-squares.
By Lemma~\ref{l:AD-cover-2}, we can cover
$J\cap Y$ by intervals:
$$
J\cap Y\ \subset\ \bigcup_{\ell=1}^L J_\ell,\quad
|J_\ell|=r,\quad
L\leq 1200C_{\AD}^2\Big({|J|\over r}+1\Big)^\delta.
$$
For each $\ell$ the projection
$$
Q_\ell:=\{y_2\in\mathbb R\colon \text{there exists }y_1\in J_\ell\text{ such that }(y_1,y_2)\in\Omega\}
$$
is contained in an interval $K_\ell$ with $|K_\ell|=T_\Omega+10r$.
Indeed, if $y_2,y'_2\in Q_\ell$, then there exist $y_1,y'_1\in J_\ell\cap J$
such that
$$
f(y_1)\leq y_2\leq g(y_1),\quad
f(y'_1)\leq y'_2\leq g(y'_1).
$$
Then
$$
y_2-y'_2\leq g(y_1)-f(y'_1)\leq g(y_1)-f(y_1)+f(y_1)-f(y'_1)
\leq T_\Omega+ 10 r
$$
where in the last inequality we use that $f\in \Lip_{10}(J)$ and $|y_1-y'_1|\leq r$.
We similarly estimate $y'_2-y_2$, and it follows that the diameter of $Q_\ell$
is bounded above by~$T_\Omega+10r$.

Next, by Lemma~\ref{l:AD-cover-2}, we can cover each $K_\ell\cap Y$ by intervals:
$$
K_\ell\cap Y\ \subset\ \bigcup_{n=1}^{N_\ell} K_{\ell n},\quad
|K_{\ell n}|=r,\quad
N_\ell\leq 1200 C_{\AD}^2\Big({T_\Omega\over r}+11\Big)^{\delta}.
$$
The set $\Omega\cap Y^2$ is covered by the $r$-squares
$J_\ell\times K_{\ell n}$. It follows that
$$
|\Omega\cap Y^2|_r\leq \sum_{\ell=1}^L N_\ell \leq 15\cdot 10^5 C_{\AD}^4
\Big({|J|\over r}+1\Big)^\delta \Big({T_\Omega\over r}+11\Big)^\delta.
$$
Together with~\eqref{e:mod-reg-int-1} this gives~\eqref{e:moderate-regular}.
\end{proof}
%%%%%%%%%%%%%%%%%%%%%%%%%%%%%%%%%%%%%%%%%%%%%%%%%%%%%%%%%%%%%%%%%%%%%%%%%%%%%%%%
The next lemma shows how to split a simple region whose bounding functions are transversal
into a union of moderate regions.
%%%%%%%%%%%%%%%%%%%%%%%%%%%%%%%%%%%%%%%%%%%%%%%%%%%%%%%%%%%%%%%%%%%%%%%%%%%%%%%%
\begin{lemma}
  \label{l:simple-split}
Let $r\in (0,1]$ and $\Omega$ be a region of the form~\eqref{e:v-simple-region} where
$J\subset [-10,10]$, $f,g\in \Lip_{10}(J)$, $g-f\geq r$ on~$J$, and
$f,g$ are transversal.
Then we can write $\Omega$ as a union of no more than $31\log(1/r)+150$ many
vertically moderate regions.
\end{lemma}
%%%%%%%%%%%%%%%%%%%%%%%%%%%%%%%%%%%%%%%%%%%%%%%%%%%%%%%%%%%%%%%%%%%%%%%%%%%%%%%%
\begin{proof}
For $m\in\mathbb Z$, define the set
$$
J_m:=\{y\in J\colon 2^{-m}\leq g(y)-f(y)\leq 2^{1-m}\}.
$$
Then
$$
J=\bigcup_{m=-4}^R J_m\quad\text{where }R:=\lceil\log_2(1/r)\rceil.
$$
For each $m$, the set $\{y_1\in J\colon g(y_1)-f(y_1)\in \{2^{-m},2^{1-m}\}\}$ has
at most 20 points and thus splits $J$ into no more than 21 intervals.
For each of these intervals, the inequality $2^{-m}\leq g-f\leq 2^{1-m}$ either
holds on the entire interval or fails in its interior. Thus we can write
$$
J_m=\bigcup_{k=1}^{K_m} J_{mk},\quad
0\leq K_m\leq 21,
$$
where each $J_{mk}$ is an interval or a point.
We then have
$$
\Omega=\bigcup_{m=-4}^R \bigcup_{k=1}^{K_m}\Omega_{mk}\quad\text{where }
\Omega_{mk}:=\Omega\cap \{y_1\in J_{mk}\},
$$
and each region $\Omega_{mk}$ is vertically moderate. (If $J_{mk}$ is a point and
we do not want to allow length~0 intervals, we can expand $J_{mk}$ to a tiny interval
in a way that $\Omega_{mk}$ is still vertically moderate.)
\end{proof}
%%%%%%%%%%%%%%%%%%%%%%%%%%%%%%%%%%%%%%%%%%%%%%%%%%%%%%%%%%%%%%%%%%%%%%%%%%%%%%%%
Our last lemma tightly covers a region bounded by several curves by moderate regions:
%%%%%%%%%%%%%%%%%%%%%%%%%%%%%%%%%%%%%%%%%%%%%%%%%%%%%%%%%%%%%%%%%%%%%%%%%%%%%%%%
\begin{lemma}
  \label{l:moderate-split-general}
Let $J\subset [-10,10]$ be an interval and $f_1,\dots,f_N,g_1,\dots,g_M\in\Lip_{10}(J)$,
where $N,M\leq 10$. Denote
$$
f_{\max}(y):=\max_{1\leq n\leq N}f_n(y),\quad
g_{\min}(y):=\min_{1\leq m\leq M}g_m(y).
$$
Define the region
$$
\Omega:=\{(y_1,y_2)\colon y_1\in J,\ f_{\max}(y_1)\leq y_2\leq g_{\min}(y_1)\}. 
$$
Assume that for each distinct pair of functions $f,g\in \{f_1,\dots,f_N,g_1,\dots,g_M\}$,
$f$ and~$g$ are either parallel or transversal. Fix $r\in (0,1]$. Then we can write
$$
\Omega\subset\bigcup_{\ell=1}^L \Omega_\ell,\quad
L\leq 6\cdot 10^4 \log(1/r)+3\cdot 10^5,\quad
|\Omega_\ell|_r\leq 9|\Omega|_r,
$$
where each $\Omega_\ell$ is a vertically moderate region.
\end{lemma}
%%%%%%%%%%%%%%%%%%%%%%%%%%%%%%%%%%%%%%%%%%%%%%%%%%%%%%%%%%%%%%%%%%%%%%%%%%%%%%%%
\begin{proof}
We say $y\in J$ is an \emph{intersection point}, if there exist
two distinct $f,g\in \{f_1,\dots,f_N,\break g_1,\dots,g_M\}$ such that
$f(y)=g(y)$. For $f,g$ parallel, there are no intersection points
and for $f,g$ transversal, there are at most 10 intersection points.
Thus there are at most 1900 intersection points in total.
These points split $J$ into intervals $J_1,\dots,J_P$ where
$P\leq 1901$.

For each $p\in [1,P]$, we have either $f_{\max}<g_{\min}$ or $f_{\max}>g_{\min}$
in the interior of $J_p$. In the first case we call $J_p$ \emph{full} and in the second
case we call $J_p$ \emph{empty}. The region~$\Omega$ is the union of the regions $\Omega^{(p)}:=\Omega\cap \{y_1\in J_p\}$ where $J_p$ is full and
at most 1900 single-point sets. Each single-point set is contained in an $r$-square,
which is a vertically moderate region with $r$-covering number equal to~1.

Thus it remains to consider a region of the form $\Omega^{(p)}$ where
$J_p$ is full. On $J_p$, we have $f_{\max}=f_{n(p)}$ and $g_{\min}=g_{m(p)}$
for some $n(p)$ and $m(p)$ depending on~$p$. Moreover, $f_{\max}\leq g_{\min}$ on~$J_p$.
Thus we have two cases:
\begin{enumerate}
\item $f_{\max}$ and $g_{\min}$ are parallel over $J_p$. Then the region $\Omega^{(p)}$
is vertically moderate by definition, and $|\Omega^{(p)}|_r\leq |\Omega|_r$.
\item $f_{\max}$ and $g_{\min}$ are transversal over $J_p$. We apply Lemma~\ref{l:simple-split}
to the functions $f_{\max}$ and $g_{\min}+r$ on $J_p$, which correspond to the region
$$
\widehat\Omega^{(p)}:=\{(y_1,y_2)\colon y_1\in J_p,\ f_{\max}(y_1)\leq y_2\leq g_{\min}(y_1)+r\}.
$$
Note that
$$
\Omega^{(p)}\ \subset\ \widehat\Omega^{(p)}\ \subset\ \Omega^{(p)}(r).
$$
We see that $\widehat\Omega^{(p)}$ can be written as a union of no more than $31\log(1/r)+150$ many vertically moderate regions $\widehat \Omega^{(p)}_1,\dots,\widehat\Omega^{(p)}_{Q_p}$. These regions
cover $\Omega^{(p)}$, and for each $q$ we have
$$
|\widehat\Omega^{(p)}_q|_r\leq |\Omega^{(p)}(r)|_r\leq 9|\Omega^{(p)}(r)|_{3r}\leq 9|\Omega^{(p)}|_r\leq 9|\Omega|_r.
$$
Here in the second inequality we use~\eqref{e:r-cover-monotone} and in the third inequality
we use~\eqref{e:X-r-nbhd}.
\end{enumerate}
\end{proof}
%%%%%%%%%%%%%%%%%%%%%%%%%%%%%%%%%%%%%%%%%%%%%%%%%%%%%%%%%%%%%%%%%%%%%%%%%%%%%%%%

%%%%%%%%%%%%%%%%%%%%%%%%%%%%%%%%%%%%%%%%%%%%%%%%%%%%%%%%%%%%%%%%%%%%%%%%%%%%%%%%
%%%%%%%%%%%%%%%%%%%%%%%%%%%%%%%%%%%%%%%%%%%%%%%%%%%%%%%%%%%%%%%%%%%%%%%%%%%%%%%%
\section{Reduction to an additive energy bound}
\label{s:reductions}

In this section we reduce Theorem~\ref{t:FUP-intro-general} to Proposition~\ref{l:additive-bound},
which is an additive energy bound.

%%%%%%%%%%%%%%%%%%%%%%%%%%%%%%%%%%%%%%%%%%%%%%%%%%%%%%%%%%%%%%%%%%%%%%%%%%%%%%%%
\subsection{Basic reductions}

We first reduce Theorem~\ref{t:FUP-intro-general} to a localized version
with the phase of a particular form. For this reduction it helps to make
%%%%%%%%%%%%%%%%%%%%%%%%%%%%%%%%%%%%%%%%%%%%%%%%%%%%%%%%%%%%%%%%%%%%%%%%%%%%%%%%
\begin{definition}
  \label{d:FUP}
Let $\mathcal U\subset\mathbb R^2$ be an open set,
$\Theta\in C^\infty(\mathcal U;\mathbb R)$, and
$\delta,\rho,\beta\in [0,1]$. We say that
\emph{$\FUP(\delta,\rho,\beta)$ holds for $\Theta$ on~$\mathcal U$} if
for each $\chi\in \CIc(\mathcal U)$ and each $C_{\AD}\geq 1$
there exists a constant $C=C(\delta,\rho,\beta,\Theta,\chi,C_{\AD})$ such that
for all $h\in (0,1]$ and all sets $X,Y\subset \mathbb R$ which are $\delta$-dimensional
AD regular on scales $C_{\AD}h^\rho$ to~1 with constant $C_{\AD}$, we have
\begin{equation}
  \label{e:FUP-defined-again}
\|\indic_X\mathcal B(h)\indic_Y\|_{L^2(\mathbb R)\to L^2(\mathbb R)}
\leq Ch^\beta,
\end{equation}
where the operator $\mathcal B(h)$ is defined in~\eqref{e:B-defined-general}.
\end{definition}
%%%%%%%%%%%%%%%%%%%%%%%%%%%%%%%%%%%%%%%%%%%%%%%%%%%%%%%%%%%%%%%%%%%%%%%%%%%%%%%%
In terms of Definition~\ref{d:FUP}, to show Theorem~\ref{t:FUP-intro-general}
it suffices to prove that $\FUP(\delta,\rho,\beta)$ holds for $\Theta$
on~$\mathcal U$, whenever $\Theta$ satisfies~\eqref{e:phase-nondegenerate}
and~\eqref{e:phase-nonlinear} and $\delta,\rho,\beta$ satisfy~\eqref{e:delta-beta-rho}.

We next give a local version of Definition~\ref{d:FUP}:
%%%%%%%%%%%%%%%%%%%%%%%%%%%%%%%%%%%%%%%%%%%%%%%%%%%%%%%%%%%%%%%%%%%%%%%%%%%%%%%%
\begin{definition}
  \label{d:FUP-local}
Let $\mathcal U\subset\mathbb R^2$ be an open set, $\Theta\in C^\infty(\mathcal U;\mathbb R)$,
$\delta,\rho,\beta\in [0,1]$, and $(x_0,y_0)\in \mathcal U$. We say that
\emph{$\FUP(\delta,\rho,\beta)$ holds for $\Theta$ near $(x_0,y_0)$}
if there exists a neighborhood $\mathcal U_0$ of $(x_0,y_0)$
such that $\FUP(\delta,\rho,\beta)$ holds for $\Theta$ on $\mathcal U_0$.
\end{definition}
%%%%%%%%%%%%%%%%%%%%%%%%%%%%%%%%%%%%%%%%%%%%%%%%%%%%%%%%%%%%%%%%%%%%%%%%%%%%%%%%
Taking a partition of unity for the amplitude~$\chi$, we see that Theorem~\ref{t:FUP-intro-general}
reduces to
%%%%%%%%%%%%%%%%%%%%%%%%%%%%%%%%%%%%%%%%%%%%%%%%%%%%%%%%%%%%%%%%%%%%%%%%%%%%%%%%
\begin{proposition}
  \label{l:FUP-reduced-1}
Let $\mathcal U\subset\mathbb R^2$ be an open set, $\Theta\in C^\infty(\mathcal U;\mathbb R)$
satisfy~\eqref{e:phase-nondegenerate} and~\eqref{e:phase-nonlinear},
and $\delta,\rho,\beta\in [0,1]$ satisfy~\eqref{e:delta-beta-rho}.
Then for each $(x_0,y_0)\in \mathcal U$, $\FUP(\delta,\rho,\beta)$ holds
for $\Theta$ near $(x_0,y_0)$.
\end{proposition}
%%%%%%%%%%%%%%%%%%%%%%%%%%%%%%%%%%%%%%%%%%%%%%%%%%%%%%%%%%%%%%%%%%%%%%%%%%%%%%%%

%%%%%%%%%%%%%%%%%%%%%%%%%%%%%%%%%%%%%%%%%%%%%%%%%%%%%%%%%%%%%%%%%%%%%%%%%%%%%%%%
\subsubsection{Gauge transformations}

We next study local gauge transformations of the phase function which preserve FUP:
%%%%%%%%%%%%%%%%%%%%%%%%%%%%%%%%%%%%%%%%%%%%%%%%%%%%%%%%%%%%%%%%%%%%%%%%%%%%%%%%
\begin{lemma}
  \label{l:FUP-gauge}
Let $\mathcal U_1,\mathcal U_2$ be open sets and $\Theta_j\in C^\infty(\mathcal U_j;\mathbb R)$,
$(x_j,y_j)\in \mathcal U_j$ for $j=1,2$. Fix $\delta,\rho,\beta\in [0,1]$.
Assume that the phase functions $\Theta_1$ near $(x_1,y_1)$ and $\Theta_2$
near $(x_2,y_2)$
are related by the following local gauge transformation:
there exist intervals $\mathcal V,\mathcal W\subset\mathbb R$ with
$x_1\in \mathcal V^\circ,y_1\in\mathcal W^\circ$, a constant $c\neq 0$, and functions
$$
\begin{aligned}
a,\varphi\in C^\infty(\mathcal V;\mathbb R),
&\quad
b,\psi\in C^\infty(\mathcal W;\mathbb R),
\\
\varphi(x_1)=x_2,\quad
\psi(y_1)=y_2,
&\quad
\varphi',\psi'\neq 0,
\\
\mathcal V\times\mathcal W\subset\mathcal U_1,
&\quad
\varphi(\mathcal V)\times\psi(\mathcal W)\subset \mathcal U_2
\end{aligned}
$$
such that
\begin{equation}
\label{e:FUP-gauge}
\Theta_1(x,y)=c\,\Theta_2(\varphi(x),\psi(y))+a(x)+b(y)\quad\text{for all }
(x,y)\in \mathcal V\times\mathcal W.
\end{equation}
Then $\FUP(\delta,\rho,\beta)$ holds for $\Theta_1$ near $(x_1,y_1)$
if and only if it holds for $\Theta_2$ near $(x_2,y_2)$.
\end{lemma}
%%%%%%%%%%%%%%%%%%%%%%%%%%%%%%%%%%%%%%%%%%%%%%%%%%%%%%%%%%%%%%%%%%%%%%%%%%%%%%%%
\begin{proof}
We assume that $\FUP(\delta,\rho,\beta)$ holds for $\Theta_2$ near $(x_2,y_2)$
and show that it holds for $\Theta_1$ near $(x_1,y_1)$; since gauge equivalence
is a symmetric relation, the other direction will follow as well.
Shrinking $\mathcal V,\mathcal W$ if needed, we may assume that
$\FUP(\delta,\rho,\beta)$ holds for $\Theta_2$ on the interior $\varphi(\mathcal V^\circ)\times\psi(\mathcal W^\circ)$.
We will show that $\FUP(\delta,\rho,\beta)$ holds for $\Theta_1$ on~$\mathcal V^\circ\times\mathcal W^\circ$.

Take any $\chi_1\in \CIc(\mathcal V^\circ\times\mathcal W^\circ)$ and let $\mathcal B_1(h)$ be defined
by~\eqref{e:B-defined-general} using $\Theta_1,\chi_1$. Assume that $X,Y\subset\mathbb R$
are $\delta$-dimensional AD regular sets on scales $C_{\AD}h^{\rho}$ to~1 with constant $C_{\AD}$.
By~\eqref{e:FUP-gauge} we have
$$
\mathcal B_1(h)f(x)=(2\pi h)^{-\frac12} \int_{\mathbb R}e^{{ic\over h}\Theta_2(\varphi(x),\psi(y))}
e^{{i\over h}a(x)}e^{{i\over h}b(y)}\chi_1(x,y)f(y)\,dy.
$$
Assume that $c>0$. Making the change of variables $\tilde x=\varphi(x),\tilde y=\psi(y)$, we see that
$$
\mathcal B_1(h)f=c^{-\frac12} e^{{i\over h}a}\varphi^* \mathcal B_2(h/c)\psi^{-*}e^{{i\over h}b}|\psi'|^{-1}f
$$
where $\mathcal B_2(h)$ is defined by~\eqref{e:B-defined-general} using $\Theta_2$
and $\chi_2(\tilde x,\tilde y):=\chi_1(\varphi^{-1}(\tilde x),\psi^{-1}(\tilde y))$.
Here $\varphi^*u=u\circ\varphi$ and $\psi^{-*}f=f\circ\psi^{-1}$. It follows that
for some constant $C_{\varphi,\psi,c}$, we have
$$
\|\indic_X\mathcal B_1(h)\indic_Y\|_{L^2\to L^2}
 \leq C_{\varphi,\psi,c} \|\indic_{\varphi(X)}\mathcal B_2(h/c)
 \indic_{\psi(Y)}\|_{L^2\to L^2}.
$$
Now the norm on the right-hand side can be estimated using that $\FUP(\delta,\rho,\beta)$
holds for $\Theta_2$ on $\varphi(\mathcal V^\circ)\times \psi(\mathcal W^\circ)$.
The required AD regularity for $\varphi(X),\psi(Y)$ follows from Lemmas~\ref{l:AD-diffeo} and~\ref{l:AD-upper}.

If $c<0$ instead, then the argument proceeds in a similar way, except we relate
$\mathcal B_1(h)$ to the adjoint of $\mathcal B_2(h/c)$.
\end{proof}
%%%%%%%%%%%%%%%%%%%%%%%%%%%%%%%%%%%%%%%%%%%%%%%%%%%%%%%%%%%%%%%%%%%%%%%%%%%%%%%%
The next lemma gives a normal form for the local equivalence class under gauge transformations of phase
functions satisfying the nondegeneracy condition~\eqref{e:phase-nondegenerate} and the
nonlinearity condition~\eqref{e:phase-nonlinear}:
%%%%%%%%%%%%%%%%%%%%%%%%%%%%%%%%%%%%%%%%%%%%%%%%%%%%%%%%%%%%%%%%%%%%%%%%%%%%%%%%
\begin{lemma}
  \label{l:gauge-reduce}
Assume that $\mathcal U_1\subset\mathbb R^2$ is an open set and
$\Theta_1\in C^\infty(\mathcal U_1;\mathbb R)$ satisfies~\eqref{e:phase-nondegenerate}
and~\eqref{e:phase-nonlinear}. Fix any $(x_1,y_1)\in \mathcal U_1$.
Then $\Theta_1$ near $(x_1,y_1)$ is related by a local gauge transformation of the form~\eqref{e:FUP-gauge} to some function $\Theta_2$ near $(0,0)$, where $\Theta_2\in C^\infty(\mathcal U_2)$ has the following Taylor expansion at $(0,0)$:
\begin{equation}
  \label{e:gauge-taylor}
\Theta_2(x,y)=xy+x^2y^2+\mathcal O(|(x,y)|^5)\quad\text{as }(x,y)\to (0,0).
\end{equation}
\end{lemma}
%%%%%%%%%%%%%%%%%%%%%%%%%%%%%%%%%%%%%%%%%%%%%%%%%%%%%%%%%%%%%%%%%%%%%%%%%%%%%%%%
\begin{proof}
Shifting $x,y$, we may assume that $(x_1,y_1)=(0,0)$.
Subtracting $\Theta_1(x,0)+\Theta_1(0,y)-\Theta_1(0,0)$ from~$\Theta_1$,
and dividing the result by the constant $\partial_x\partial_y\Theta_1(0,0)$,
we may pass to a gauge equivalent function $\Theta_3$ satisfying
$$
\Theta_3(x,0)=\Theta_3(0,y)=0,\quad
\partial_x\partial_y\Theta_3(0,0)=1.
$$
The function $\Theta_3$ has the following Taylor expansion at $(0,0)$:
$$
\Theta_3(x,y)=xy+c_{21}x^2y+c_{12}xy^2+c_{31}x^3y+c_{22}x^2y^2+c_{13}xy^3
+\mathcal O(|(x,y)|^5).
$$
We make the change of variables
$$
\Theta_4(x,y):=\Theta_3(x-c_{21} x^2+(2c_{21}^2-c_{31}) x^3,y-c_{12} y^2+(2c_{12}^2-c_{13}) y^3)
$$
and the function $\Theta_4$ has the following Taylor expansion at $(0,0)$:
$$
\Theta_4(x,y)=xy+(c_{22}-c_{12}c_{21})x^2y^2+\mathcal O(|(x,y)|^5).
$$
By the nonlinear phase condition~\eqref{e:phase-nonlinear}, the coefficient $\tilde c:=c_{22}-c_{12}c_{21}$ is nonzero. We make the final gauge transformation
$$
\Theta_2(x,y)=\tilde c\,\Theta_4(x/\tilde c,y)
$$
and $\Theta_2$ has the Taylor expansion~\eqref{e:gauge-taylor}.
\end{proof}
%%%%%%%%%%%%%%%%%%%%%%%%%%%%%%%%%%%%%%%%%%%%%%%%%%%%%%%%%%%%%%%%%%%%%%%%%%%%%%%%

%%%%%%%%%%%%%%%%%%%%%%%%%%%%%%%%%%%%%%%%%%%%%%%%%%%%%%%%%%%%%%%%%%%%%%%%%%%%%%%%
\subsubsection{A reduced form of FUP}

Combining Lemmas~\ref{l:FUP-gauge} and~\ref{l:gauge-reduce}, we see that
Proposition~\ref{l:FUP-reduced-1} (and thus Theorem~\ref{t:FUP-intro-general}) reduces to showing that
for each $\Theta\in C^\infty(\mathcal U;\mathbb R)$ satisfying~\eqref{e:gauge-taylor}
and any $\delta,\rho,\beta\in[0,1]$ satisfying~\eqref{e:delta-beta-rho},
$\FUP(\delta,\rho,\beta)$ holds for $\Theta$ near~$(0,0)$.

To make the localization near $0$ explicit, take a small ($h$-independent) parameter $r_0>0$
and define the intervals
\begin{equation}
  \label{e:I-0-def}
I_0:=[-r_0,r_0],\quad
I'_0:=[-0.4r_0,0.4r_0],\quad
I''_0:=[-0.1r_0,0.1r_0].
\end{equation}
It suffices to show that for any $\Theta$ satisfying~\eqref{e:gauge-taylor}
and $r_0>0$ small enough, we have the estimate~\eqref{e:FUP-defined-again}
when $\supp\chi\subset (I''_0)^2$. The estimate~\eqref{e:FUP-defined-again}
stays the same if we replace $X,Y$ by $X\cap I''_0$, $Y\cap I''_0$.
By Lemma~\ref{l:AD-split} and Lemma~\ref{l:AD-upper}, we can write
$$
X\cap I''_0\subset \bigsqcup_{n=1}^N X_n
$$
where $N$ is bounded and each $X_n$ is contained in a $0.1r_0$-interval
and is $\delta$-dimensional AD regular on scales~$C_{\AD}h^{\rho}$ to~1.
We may assume that each $X_n$ is contained in $I'_0$. 
Decomposing the set $Y$ in a similar way, we can reduce to
the case when $X,Y\subset I'_0$. 

The above discussion reduces Theorem~\ref{t:FUP-intro-general} to the following
%%%%%%%%%%%%%%%%%%%%%%%%%%%%%%%%%%%%%%%%%%%%%%%%%%%%%%%%%%%%%%%%%%%%%%%%%%%%%%%%
\begin{proposition}
  \label{l:FUP-reduced-2}
Let $\mathcal U\subset\mathbb R^2$ be an open set containing $(0,0)$
and $\Theta\in C^\infty(\mathcal U;\mathbb R)$ have the following Taylor expansion
at $(0,0)$:
\begin{equation}
  \label{e:phase-cooked}
\Theta(x,y)=xy+x^2y^2+\mathcal O(|(x,y)|^5)\quad\text{as }(x,y)\to (0,0).
\end{equation}
Assume that $\delta,\rho,\beta\in [0,1]$ satisfy~\eqref{e:delta-beta-rho}.
Then there exists $r_0>0$, depending on~$\Theta$, such that for each $\chi\in\CIc(\mathcal U)$ with $\supp\chi\subset (I''_0)^2$
and each $C_{\AD}\geq 1$, there exists a constant $C$ such that the following holds.
Let $h\in (0,1]$ and $X,Y\subset I'_0$ be $\delta$-dimensional AD regular sets on scales $C_{\AD}h^\rho$ to~1
with constant $C_{\AD}$. Then
\begin{equation}
  \label{e:FUP-reduced-2}
\|\indic_X\mathcal B(h)\indic_Y\|_{L^2\to L^2}\leq Ch^\beta
\end{equation}
where $\mathcal B(h)$ is defined by~\eqref{e:B-defined-general}.
\end{proposition}
%%%%%%%%%%%%%%%%%%%%%%%%%%%%%%%%%%%%%%%%%%%%%%%%%%%%%%%%%%%%%%%%%%%%%%%%%%%%%%%%

%%%%%%%%%%%%%%%%%%%%%%%%%%%%%%%%%%%%%%%%%%%%%%%%%%%%%%%%%%%%%%%%%%%%%%%%%%%%%%%%
\subsection{Preparation for the reduction}

Before reducing Proposition~\ref{l:FUP-reduced-2} to an additive energy bound,
we need some preparatory work. We first introduce some notation.
Denote
\begin{equation}
  \label{e:vec-y}
\vec y=(y_1,y_2,y'_1,y'_2)\in \mathbb R^4,\quad
d\vec y:=dy_1dy_2dy'_1dy'_2.
\end{equation}
Given $\Theta\in C^\infty(I_0^2;\mathbb R)$, define the following function
for $x\in I_0$, $\vec y\in I_0^4$:
\begin{equation}
  \label{e:Theta-vec-def}
\vec\Theta(x,\vec y):=\Theta(x,y_1)+\Theta(x,y_2)
-\Theta(x,y'_1)-\Theta(x,y'_2).
\end{equation}
Next, for $r>0$ define the set
\begin{equation}
  \label{e:E-r-def}
\mathcal E_r(\Theta):=\{(x,\vec y)\in I_0^5\colon
|\partial_x\vec\Theta(x,\vec y)|\leq r\}.
\end{equation}

%%%%%%%%%%%%%%%%%%%%%%%%%%%%%%%%%%%%%%%%%%%%%%%%%%%%%%%%%%%%%%%%%%%%%%%%%%%%%%%%
\subsubsection{Nonstationary phase}

We will use the following nonstationary phase lemma, similar to~\cite[Lemma~5.1]{hgap}.
%%%%%%%%%%%%%%%%%%%%%%%%%%%%%%%%%%%%%%%%%%%%%%%%%%%%%%%%%%%%%%%%%%%%%%%%%%%%%%%%
\begin{lemma}
  \label{l:ibp}
Assume that $\mathcal U\subset \mathbb R$ is an open set,
$\Phi\in C^\infty(\mathcal U;\mathbb R)$, $a\in \CIc(\mathcal U)$, and
the following inequalities hold for some positive numbers $R_1, R_2,(C_j)_{j\geq 0}$:
\begin{align}
  \label{e:ibpa-1}
|\partial_x\Phi(x)|&\geq C_0^{-1}R_1\quad\text{for all }x\in\supp a;
\\
  \label{e:ibpa-2}
|\partial_x^{j+1}\Phi(x)|&\leq C_j R_1R_2^j\quad\text{for all }x\in\supp a,\
j\geq 1;
\\
  \label{e:ibpa-3}
|\partial_x^j a(x)|&\leq C_j R_2^j\quad\text{for all }x\in \mathcal U,\
j\geq 0.
\end{align}
Then we have for each $N\geq 0$
\begin{equation}
  \label{e:ibp}
\bigg|\int_{\mathbb R} e^{i\Phi(x)}a(x)\,dx\bigg|\leq C'_N\Big({R_2\over R_1}\Big)^N|\supp a|
\end{equation}
where the constant $C'_N$ depends only on $N,C_0,C_1,\dots,C_N$.
\end{lemma}
%%%%%%%%%%%%%%%%%%%%%%%%%%%%%%%%%%%%%%%%%%%%%%%%%%%%%%%%%%%%%%%%%%%%%%%%%%%%%%%%
\begin{proof}
Consider the following first order differential operator on $\mathcal U$:
$$
\mathcal L:=-i(\partial_x\Phi(x))^{-1}\partial_x.
$$
Then $e^{i\Phi}=\mathcal Le^{i\Phi}$, so we can integrate by parts $N$ times to get
$$
\int_{\mathbb R} e^{i\Phi(x)}a(x)\,dx
=\int_{\mathbb R} \big(\mathcal L^N e^{i\Phi(x)}\big)a(x)\,dx
=\int_{\mathbb R} e^{i\Phi(x)}
(\mathcal L^t)^Na(x)\,dx
$$
where $\mathcal L^t$ is the transpose operator:
$$
\mathcal L^tf(x)=i\partial_x\bigg({f(x)\over\partial_x\Phi(x)}\bigg).
$$
It follows that
\begin{equation}
  \label{e:ibp-1}
\bigg|\int_{\mathbb R} e^{i\Phi(x)}a(x)\,dx\bigg|
\leq |\supp a|\cdot \sup_{x\in\supp a}|(\mathcal L^t)^Na(x)|.
\end{equation}
The function $(\mathcal L^t)^ Na(x)$ is a linear combination with constant
coefficients of terms of the form
$$
(\partial_x\Phi(x))^{-N-K}\Big(\prod_{k=1}^K {\partial_x^{j_k+1}\Phi(x)}\Big)\partial_x^{j_0}a(x)
$$
where $j_0\geq 0$, $j_1,\dots,j_K\geq 1$, and $j_0+j_1+\dots+j_K= N$.
By~\eqref{e:ibpa-1}--\eqref{e:ibpa-3}, the sup-norm of this expression
over $x\in\supp a$ is bounded above by
$$
C'_N R_1^{-N-K}\Big(\prod_{k=1}^K R_1R_2^{j_k}\Big)R_2^{j_0}=C'_N\Big({R_2\over R_1}\Big)^N.
$$
Together with~\eqref{e:ibp-1} this gives~\eqref{e:ibp}.
\end{proof}
%%%%%%%%%%%%%%%%%%%%%%%%%%%%%%%%%%%%%%%%%%%%%%%%%%%%%%%%%%%%%%%%%%%%%%%%%%%%%%%%

%%%%%%%%%%%%%%%%%%%%%%%%%%%%%%%%%%%%%%%%%%%%%%%%%%%%%%%%%%%%%%%%%%%%%%%%%%%%%%%%
\subsubsection{An $L^4$ estimate}

We now bound the norm of $\indic_X\mathcal B(h)\indic_Y$ in terms of the volume
of the set $\mathcal E_r$ introduced in~\eqref{e:E-r-def}. For this it is convenient
to consider the $L^\infty\to L^4$ norm; later in~\eqref{e:L2-L2} we interpolate between
this and the $L^1\to L^\infty$ norm to bound the $L^2\to L^2$ norm.

Recall the notation~\eqref{e:I-0-def}.
We start with the case when $X$ is replaced by a $2h^{\frac12}$-interval:
%%%%%%%%%%%%%%%%%%%%%%%%%%%%%%%%%%%%%%%%%%%%%%%%%%%%%%%%%%%%%%%%%%%%%%%%%%%%%%%%
\begin{lemma}
  \label{l:Linfty-L4-interval}
Let $\mathcal B(h)$ be the operator defined by~\eqref{e:B-defined-general}
with some $\Theta\in C^\infty(I_0^2;\mathbb R)$ and $\supp\chi\subset (I''_0)^2$.
Let $x_0\in I_0$, $Y\subset I_0$, and put $J:=[x_0-h^{\frac12},x_0+h^{\frac12}]$. Then
we have for all $\epsilon>0$ and $N$
\begin{equation}
  \label{e:Linfty-L4-interval}
\|\indic_J \mathcal B(h)\indic_Y\|_{L^\infty\to L^4}^4\leq
Ch^{-\frac32}\big|\{\vec y\in Y^4\colon (x_0,\vec y)\in \mathcal E_{h^{\frac12-\epsilon}}(\Theta)\}\big|+Ch^N.
\end{equation}
Here the constant $C$ depends only on $\epsilon,N,\Theta,\chi$.
\end{lemma}
%%%%%%%%%%%%%%%%%%%%%%%%%%%%%%%%%%%%%%%%%%%%%%%%%%%%%%%%%%%%%%%%%%%%%%%%%%%%%%%%
\begin{proof}
\noindent 1. Fix a cutoff function
$$
\chi\in \CIc((-2,2);[0,1]),\quad
\chi=1\quad\text{on }[-1,1]
$$
and define the function
$$
\chi_J\in\CIc(\mathbb R),\quad
\chi_J(x)=\chi\Big({x-x_0\over h^{\frac12}}\Big).
$$
Note that $\chi_J=1$ on $J$.

Take $f\in L^\infty(\mathbb R)$. Then, using the notation~\eqref{e:vec-y} and~\eqref{e:Theta-vec-def},
\begin{equation}
  \label{e:Li4i-0}
\begin{aligned}
\|\indic_J\mathcal B(h)\indic_Y f\|_{L^4}^4
&\leq
\int_{\mathbb R}\chi_J(x)|\mathcal B(h)\indic_Y f(x)|^4\,dx
\\
&=\int_{\mathbb R}\chi_J(x)\mathcal B(h)\indic_Y f(x)\mathcal B(h)\indic_Y f(x)
\overline{\mathcal B(h)\indic_Y f(x)\mathcal B(h)\indic_Y f(x)}\,dx
\\
&=\int_{Y^4}F_J(\vec y)f(y_1)f(y_2)\overline{f(y'_1)f(y'_2)}\,d\vec y
\end{aligned}
\end{equation}
where
$$
\begin{aligned}
F_J(\vec y)\,&:=
(2\pi h)^{-2}\int_{\mathbb R}e^{{i\over h}\vec\Theta(x,\vec y)}\vec\chi_J(x,\vec y)\,dx,
\\
\vec\chi_J(x,\vec y)\,&:=
\chi_J(x)\chi(x,y_1)\chi(x,y_2)
\overline{\chi(x,y'_1)\chi(x,y'_2)}.
\end{aligned}
$$

\noindent 2. We claim that for all $\vec y\in I_0^4$ and $\widetilde N$
\begin{equation}
  \label{e:Li4i-1}
|F_J(\vec y)|\leq C_{\widetilde N} h^{-\frac32}\bigg(1+{|\partial_x\vec\Theta(x_0,\vec y)|\over h^{\frac12}}\bigg)^{-\widetilde N}.
\end{equation}
For that we use Lemma~\ref{l:ibp} with
$$
\mathcal U:=I_0^{\circ},\quad
\Phi(x):={\vec\Theta(x,\vec y)\over h},\quad
a(x):=\vec\chi_J(x,\vec y).
$$
Note that $\supp a\subset \supp\chi_J\subset [x_0-2h^{\frac12},x_0+2h^{\frac12}]$,
so $|\supp a|\leq 4h^{\frac12}$ and for all $x\in\supp a$ we have
$$
|\partial_x\vec\Theta(x,\vec y)-\partial_x\vec\Theta(x_0,\vec y)|\leq \widetilde Ch^{\frac12},\quad
\widetilde C:=\max(8\sup|\partial_x^2\Theta|,1).
$$
If $|\partial_x\vec\Theta(x_0,\vec y)|\leq 2\widetilde Ch^{\frac12}$, then~\eqref{e:Li4i-1} is immediate.
Assume now that $|\partial_x\vec\Theta(x_0,\vec y)|\geq 2\widetilde Ch^{\frac12}$. We have for all $x\in \supp a$ and $j\geq 0$
$$
|\partial_x\Phi(x)|\geq {|\partial_x\vec\Theta(x_0,\vec y)|\over 2h},\quad
|\partial_x^{j+1}\Phi(x)|\leq {C_j\over h},\quad
|\partial_x^j a(x)|\leq C_j h^{-\frac12 j}.
$$
The bounds~\eqref{e:ibpa-1}--\eqref{e:ibpa-3} hold with
$$
R_1:={|\partial_x\vec\Theta(x_0,\vec y)|\over h},\quad
R_2:=h^{-\frac12}.
$$
Then~\eqref{e:ibp} gives~\eqref{e:Li4i-1} as needed.

\noindent 3. By~\eqref{e:Li4i-0}, we have
$$
\begin{aligned}
{\|\indic_J\mathcal B(h)\indic_Y f\|_{L^4}^4\over \|f\|_{L^\infty}^4}
\leq 
\int_{Y^4}|F_J(\vec y)|\,d\vec y
\leq &\,
\big|\{\vec y\in Y^4\colon |\partial_x\vec\Theta(x_0,\vec y)|\leq h^{\frac12-\epsilon}\}\big|\sup_{\vec y\in I_0^4}|F_J(\vec y)|
\\
&+ \sup_{\vec y\in I_0^4\colon |\partial_x\vec\Theta(x_0,\vec y)|\geq h^{\frac12-\epsilon}}
|F_J(\vec y)|.
\end{aligned}
$$
By~\eqref{e:Li4i-1}, we see that the first supremum on the right-hand side
is bounded above by $Ch^{-\frac32}$ and the second one is bounded by
$C_{\widetilde N}h^{\epsilon\widetilde N-\frac32}\leq Ch^N$,
where we take $\widetilde N\geq\epsilon^{-1}(N+\frac32)$. This finishes the proof of~\eqref{e:Linfty-L4-interval}.
\end{proof}
%%%%%%%%%%%%%%%%%%%%%%%%%%%%%%%%%%%%%%%%%%%%%%%%%%%%%%%%%%%%%%%%%%%%%%%%%%%%%%%%
We now estimate the norm of $\indic_X\mathcal B(h)\indic_Y$ by summing Lemma~\ref{l:Linfty-L4-interval}
over a collection of $h^{\frac12}$-intervals covering~$X$:
%%%%%%%%%%%%%%%%%%%%%%%%%%%%%%%%%%%%%%%%%%%%%%%%%%%%%%%%%%%%%%%%%%%%%%%%%%%%%%%%
\begin{lemma}
  \label{l:Linfty-L4}
Let $X,Y\subset I_0$. Then
for all $\epsilon>0$ and~$N$
\begin{align}
  \label{e:Linfty-L4}
\|\indic_X\mathcal B(h)\indic_Y\|_{L^\infty\to L^4}
&\leq
Ch^{-\frac12}|(X(h^{\frac12})\times Y^4)\cap \mathcal E_{h^{\frac12-\epsilon}}(\Theta)|^{\frac14}
+Ch^N,
\\
  \label{e:L2-L2}
\|\indic_X\mathcal B(h)\indic_Y\|_{L^2\to L^2}
&\leq
Ch^{-\frac12}|X|^{\frac38}
|(X(h^{\frac12})\times Y^4)\cap \mathcal E_{h^{\frac12-\epsilon}}(\Theta)|^{\frac18}+Ch^N.
\end{align}
Here the constant $C$ depends only on $\epsilon,N,\Theta,\chi$.
\end{lemma}
%%%%%%%%%%%%%%%%%%%%%%%%%%%%%%%%%%%%%%%%%%%%%%%%%%%%%%%%%%%%%%%%%%%%%%%%%%%%%%%%
\begin{proof}
1. Let $x_1,\dots,x_L\in X$ be a maximal $\frac12h^{\frac12}$-separated set
and define
$$
J_\ell:=[x_\ell-\tfrac12h^{\frac12},x_\ell+\tfrac12h^{\frac12}].
$$
Then $X\subset\bigcup_{\ell=1}^L J_\ell$ and thus
$$
\|\indic_X\mathcal B(h)\indic_Y\|_{L^\infty\to L^4}^4
\leq \sum_{\ell=1}^L\|\indic_{J_\ell}\mathcal B(h)\indic_Y\|_{L^\infty\to L^4}^4.
$$
For each $x\in J_\ell\cap I_0$, we have $J_\ell\subset [x-h^{\frac12},x+h^{\frac12}]$ and thus
by Lemma~\ref{l:Linfty-L4-interval}
$$
\|\indic_{J_\ell} \mathcal B(h)\indic_Y\|_{L^\infty\to L^4}^4\leq
Ch^{-\frac32}\big|\{\vec y\in Y^4\colon (x,\vec y)\in \mathcal E_{h^{\frac12-\epsilon}}(\Theta)\}\big|
+Ch^{4N+1}.
$$
Integrating over $x\in J_\ell\cap I_0$, we get
$$
\|\indic_{J_\ell} \mathcal B(h)\indic_Y\|_{L^\infty\to L^4}^4\leq
Ch^{-2}\big|
(J_\ell\times Y^4)\cap \mathcal E_{h^{\frac12-\epsilon}}(\Theta)
\big|+Ch^{4N+1}.
$$
Summing over $\ell$, and using that each point lies in at most 3 intervals $J_\ell$ and
$\bigcup_{\ell=1}^L J_\ell\subset X(h^{\frac12})$,
we then have
$$
\|\indic_X\mathcal B(h)\indic_Y\|_{L^\infty\to L^4}^4
\leq
Ch^{-2}\big|
(X(h^{\frac12})\times Y^4)\cap \mathcal E_{h^{\frac12-\epsilon}}(\Theta)
\big|
+Ch^{4N}
$$
which implies~\eqref{e:Linfty-L4}.

\noindent 2. By H\"older's inequality, we have the bounds
\begin{align}
  \label{e:Linfty-L2}
\|\indic_X\mathcal B(h)\indic_Y\|_{L^\infty\to L^2}
&\leq
|X|^{\frac14}\,\|\indic_X\mathcal B(h)\indic_Y\|_{L^\infty\to L^4},
\\
  \label{e:L1-L2}
\|\indic_X\mathcal B(h)\indic_Y\|_{L^1\to L^2}
&\leq
|X|^{\frac12}\,\|\indic_X\mathcal B(h)\indic_Y\|_{L^1\to L^\infty}.
\end{align}
Interpolating between these gives
\begin{equation}
  \label{e:bizzy}
\|\indic_X\mathcal B(h)\indic_Y\|_{L^2\to L^2}
\leq
|X|^{\frac38}\,
\|\indic_X\mathcal B(h)\indic_Y\|_{L^\infty\to L^4}^{\frac12}\,
\|\indic_X\mathcal B(h)\indic_Y\|_{L^1\to L^\infty}^{\frac12}.
\end{equation}
From the definition~\eqref{e:B-defined-general} of~$\mathcal B(h)$, we have
$$
\|\mathcal B(h)\|_{L^1\to L^\infty}\leq Ch^{-\frac12}.
$$
Together with~\eqref{e:Linfty-L4} and~\eqref{e:bizzy} this gives~\eqref{e:L2-L2}.
\end{proof}
%%%%%%%%%%%%%%%%%%%%%%%%%%%%%%%%%%%%%%%%%%%%%%%%%%%%%%%%%%%%%%%%%%%%%%%%%%%%%%%%

%%%%%%%%%%%%%%%%%%%%%%%%%%%%%%%%%%%%%%%%%%%%%%%%%%%%%%%%%%%%%%%%%%%%%%%%%%%%%%%%
\subsection{Proof of the reduction}

We now state the additive energy bound to which Proposition~\ref{l:FUP-reduced-2} will be reduced.
It is proved in~\S\ref{s:additive-bound}.
%%%%%%%%%%%%%%%%%%%%%%%%%%%%%%%%%%%%%%%%%%%%%%%%%%%%%%%%%%%%%%%%%%%%%%%%%%%%%%%%
\begin{proposition}
  \label{l:additive-bound}
Let $\mathcal U\subset\mathbb R^2$ be an open set containing $(0,0)$
and $\Theta\in C^\infty(\mathcal U;\mathbb R)$ have the Taylor expansion~\eqref{e:phase-cooked}.
Then there exists $r_0>0$, depending on~$\Theta$, such that, using the notation~\eqref{e:I-0-def}, we have
$I_0^2\subset\mathcal U$ and the following holds. Assume that
$\delta\in (0,\frac23]$, $r\leq 1$ is a dyadic number, and
$X,Y\subset I'_0$ are $\delta$-dimensional AD regular sets on scales~$r$ to~1
with constant $C_{\AD}$. Then we have for each $\epsilon>0$
\begin{equation}
\label{e:additive-bound}
|(X\times Y^4)\cap \mathcal E_r(\Theta)|\leq C r^{5-\frac72\delta-\epsilon}
\end{equation}
where the constant $C$ depends only on~$\Theta,\delta,\epsilon,C_{\AD}$ and
$\mathcal E_r(\Theta)$ was defined in~\eqref{e:E-r-def}.
\end{proposition}
%%%%%%%%%%%%%%%%%%%%%%%%%%%%%%%%%%%%%%%%%%%%%%%%%%%%%%%%%%%%%%%%%%%%%%%%%%%%%%%%
\begin{remark}
We note that the trivial bound is
$$
|(X\times Y^4)\cap \mathcal E_r(\Theta)|\leq Cr|X|\cdot|Y|^3\leq C r^{5-4\delta}
$$
so Proposition~\ref{l:additive-bound} gives a gain of $r^{\frac12\delta-}$.
\end{remark}
%%%%%%%%%%%%%%%%%%%%%%%%%%%%%%%%%%%%%%%%%%%%%%%%%%%%%%%%%%%%%%%%%%%%%%%%%%%%%%%%

We are ready to finish the proof of the reduction:
%%%%%%%%%%%%%%%%%%%%%%%%%%%%%%%%%%%%%%%%%%%%%%%%%%%%%%%%%%%%%%%%%%%%%%%%%%%%%%%%
\begin{proof}[Proof of Proposition~\ref{l:FUP-reduced-2}]
Recall that $X,Y\subset I'_0$ are $\delta$-dimensional
AD regular on scales $C_{\AD}h^\rho$ to~1 with constant $C_{\AD}$. By Lemma~\ref{l:AD-cover} and~\eqref{e:volume-covering-M}, we get
\begin{align}
  \label{e:reduction-XY-1}
|X| &\leq 12C_{\AD}^3 h^{\rho(1-\delta)},
\\
  \label{e:reduction-XY-2}
\mathbf M_{h^{\frac12}}(Y) &\leq 12C_{\AD}^3 h^{\frac12\delta+\rho(1-\delta)}.
\end{align}
By Lemma~\ref{l:Linfty-L4} with $N:=1$, and recalling from~\eqref{e:delta-beta-rho} that $\rho=1-\epsilon$, we then have
\begin{equation}
  \label{e:reduction-int-1}
\|\indic_X\mathcal B(h)\indic_Y\|_{L^2\to L^2}
\leq
Ch^{-\frac18-\frac38(\delta+\epsilon)}
|(X(h^{\frac12})\times Y^4)\cap \mathcal E_{h^{\frac12-\epsilon}}(\Theta)|^{\frac18}+Ch.
\end{equation}
Let $r$ be a dyadic number such that
$$
\tfrac14 r\leq h^{\frac12-\epsilon}\leq \tfrac12 r.
$$
Denote
$$
\widehat X:=X(2h^{\frac12}),\quad
\widehat Y:=Y(h^{\frac12}).
$$
The neighborhood $((X(h^{\frac12})\times Y^4)\cap \mathcal E_{h^{\frac12-\epsilon}}(\Theta))(h^{\frac12})$
is contained in $(\widehat X\times \widehat Y^4)\cap \mathcal E_r(\Theta)$. Therefore
\begin{equation}
  \label{e:reduction-int-2}
\begin{aligned}
|(X(h^{\frac12})\times Y^4)\cap \mathcal E_{h^{\frac12-\epsilon}}(\Theta)|
&\leq h^{\frac12}\mathbf M_{h^{\frac12}}(Y)^4|(X(h^{\frac12})\times Y^4)\cap \mathcal E_{h^{\frac12-\epsilon}}(\Theta)|_{h^{\frac12}}
\\
&\leq 32h^{-2}\mathbf M_{h^{\frac12}}(Y)^4|(\widehat X\times \widehat Y^4)\cap \mathcal E_r(\Theta)|
\\
&\leq Ch^{2(1-\delta)-4\epsilon}|(\widehat X\times \widehat Y^4)\cap \mathcal E_r(\Theta)|.
\end{aligned}
\end{equation}
Here in the first inequality we use~\eqref{e:volume-covering-M}. 
In the second inequality we use~\eqref{e:volume-nbhd}.
In the last inequality we use~\eqref{e:reduction-XY-2}.

By Lemma~\ref{l:AD-nbhd}, the sets $\widehat X,\widehat Y\subset I'_0$ are $\delta$-dimensional
AD regular on scales $r$ to~1 with constant $4C_{\AD}$.
Applying Proposition~\ref{l:additive-bound} to these sets,
we see that
\begin{equation}
  \label{e:reduction-int-3}
|(\widehat X\times \widehat Y^4)\cap \mathcal E_r(\Theta)|
\leq C h^{\frac52-\frac74\delta-5.5\epsilon}.
\end{equation}
Putting together~\eqref{e:reduction-int-1}, \eqref{e:reduction-int-2}, and~\eqref{e:reduction-int-3},
we get
$$
\|\indic_X\mathcal B(h)\indic_Y\|_{L^2\to L^2}\leq 
Ch^{\frac{7}{16}-\frac{27}{32}\delta-\frac{25}{16}\epsilon}
+Ch.
$$
Recalling~\eqref{e:delta-beta-rho}, we see that this gives~\eqref{e:FUP-reduced-2}.
\end{proof}
%%%%%%%%%%%%%%%%%%%%%%%%%%%%%%%%%%%%%%%%%%%%%%%%%%%%%%%%%%%%%%%%%%%%%%%%%%%%%%%%

%%%%%%%%%%%%%%%%%%%%%%%%%%%%%%%%%%%%%%%%%%%%%%%%%%%%%%%%%%%%%%%%%%%%%%%%%%%%%%%%
%%%%%%%%%%%%%%%%%%%%%%%%%%%%%%%%%%%%%%%%%%%%%%%%%%%%%%%%%%%%%%%%%%%%%%%%%%%%%%%%
\section{Proof of the additive energy bound}
\label{s:additive-bound}

In this section, we prove Proposition~\ref{l:additive-bound}.

%%%%%%%%%%%%%%%%%%%%%%%%%%%%%%%%%%%%%%%%%%%%%%%%%%%%%%%%%%%%%%%%%%%%%%%%%%%%%%%%
\subsection{A theorem from projection theory}

We will use a theorem in projection theory due to O'Regan--Wu--Yi~\cite{ORegan-Wu-Yi}. In this subsection we state this theorem and give a slight modification of it.

%%%%%%%%%%%%%%%%%%%%%%%%%%%%%%%%%%%%%%%%%%%%%%%%%%%%%%%%%%%%%%%%%%%%%%%%%%%%%%%%
\subsubsection{Notation}

We first need to review some notation and definitions from~\cite{ORegan-Wu-Yi}:%
\footnote{The paper~\cite{ORegan-Wu-Yi} uses the letter $\delta$ to denote the small scale at
which their results are applied. In our applications to hyperbolic surfaces, $\delta$ 
denotes the dimension of the limit set,
so we use the letter $r$ to denote the small scale instead of~$\delta$.
Also, we use $\#(\bullet)$ for the cardinality of a set and $|\bullet|$ for its Lebesgue
measure, while~\cite{ORegan-Wu-Yi} uses $|\bullet|$ for cardinality.}
\begin{itemize}
\item Let $r\in (0,1]$ and $\mathbb P\subset \mathcal D_r(\mathbb R^d)$ be a set of grid $r$-cubes.
Let also $s\geq 0$ (the dimension) and $C>0$ (the regularity constant). We say
that $\mathbb P$ is an \emph{$(r,s,C)$-KT cube family}, if 
$$
\#\{\mathbf p\in\mathbb P\colon \mathbf p\cap Q\neq\emptyset\}\leq
C\Big({r_1\over r}\Big)^s\quad\text{for all $r_1$-cubes $Q\subset\mathbb R^d$,
$r_1\in [r,1]$.}
$$
(Here KT stands for Katz--Tao.)
\item If $I\subset \mathbb R$ is a (compact) interval, we use the following norm on the space $C^2(I)$:
$$
\|F\|_{C^2(I)}:=\max_{x\in I}\sum_{k=0}^2 |F^{(k)}(x)|.
$$
(In this section, the functions in $C^2(I)$ are always real-valued.)
\item For $F\in C^2(I)$, denote its graph by
$$
\Gamma_F:=\{(x,F(x))\colon x\in I\}\ \subset\ \mathbb R^2.
$$
\item For $F\in C^2(I)$ and $x\in I$, denote the 1-jet of~$F$ at~$x$ by
$$
A_x(F)=(F(x),F'(x))\ \in\ \mathbb R^2.
$$
For a set $\mathbb F\subset C^2(I)$, this definition gives a set $A_x(\mathbb F)\subset \mathbb R^2$.
\item If $\mathfrak T\geq 1$, we say a set $\mathbb F\subset C^2(I)$ is a \emph{$\mathfrak T$-transversal family on~$I$}, if
\begin{equation}
\inf_{x\in I}|A_x(F)-A_x(G)|\geq\mathfrak T^{-1}\|F-G\|_{C^2(I)}\quad\text{for all }
F,G\in \mathbb F.
\end{equation}
\item If $r\in (0,1]$, $\tau\in [0,2]$, and $C\geq 1$, we say that $\mathbb F\subset C^2(I)$ is an \emph{$(r,\tau,C)$-RKT set}, if for all $x\in I$ we have
$$
\begin{gathered}
\big|A_x(\mathbb F)\cap (I_1\times I_2)\big|_r\leq C\Big({\sqrt{r_1r_2}\over r}\Big)^\tau
\\
\quad\text{for all }r_1,r_2\in [r,1]
\quad\text{and all intervals }I_1,I_2\subset\mathbb R
\quad\text{with }|I_j|=r_j.
\end{gathered}
$$
(Here RKT stands for rectangular Katz--Tao.)
\end{itemize}

%%%%%%%%%%%%%%%%%%%%%%%%%%%%%%%%%%%%%%%%%%%%%%%%%%%%%%%%%%%%%%%%%%%%%%%%%%%%%%%%
\subsubsection{Statement of the projection theorem}

We are now ready to state the result that we use, which is
\cite[Theorem~4.4]{ORegan-Wu-Yi}:
%%%%%%%%%%%%%%%%%%%%%%%%%%%%%%%%%%%%%%%%%%%%%%%%%%%%%%%%%%%%%%%%%%%%%%%%%%%%%%%%
\begin{theorem}
  \label{t:OWY}
Let $s\in (0,1]$, $C_1,C_2\geq 1$, and $\mathfrak T\geq 1$. For any $\epsilon>0$,
there exists $r'_0=r'_0(\mathfrak T,s,\epsilon)>0$ such that the following holds for
all dyadic $r\in (0,r'_0]$. Let $\mathbb F\subset B_{C^2([-2,2])}(0,1)$ be an $r$-separated
$\mathfrak T$-transversal family on $[-2,2]$. Assume that
\begin{itemize}
\item $\mathbb F$ is an $(r,2s,C_1)$-RKT set;
\item for each $F\in\mathbb F$, assume that there exists an $(r,\sigma,C_2)$-KT cube family
$$
\mathbb P(F)\subset \{\mathbf p\in \mathcal D_r([-1,1]^2)\colon\mathbf  p\cap \Gamma_F\neq\emptyset\}.
$$ 
Here $\sigma=\min(2s,2-2s)$.
\end{itemize}
Write $\mathbb P:=\bigcup_{F\in\mathbb F}\mathbb P(F)$, then
\begin{equation}
  \label{e:OWY}
\sum_{F\in\mathbb F} \#(\mathbb P(F))\leq C_1^{\frac23}C_2^{\frac13}r^{-\frac23 s-\epsilon}
\#(\mathbb P)^{\frac23 }\#(\mathbb F)^{\frac13}.
\end{equation}
\end{theorem}
%%%%%%%%%%%%%%%%%%%%%%%%%%%%%%%%%%%%%%%%%%%%%%%%%%%%%%%%%%%%%%%%%%%%%%%%%%%%%%%%

%%%%%%%%%%%%%%%%%%%%%%%%%%%%%%%%%%%%%%%%%%%%%%%%%%%%%%%%%%%%%%%%%%%%%%%%%%%%%%%%
\subsubsection{Point-line incidences and a restatement of the projection theorem}

Fix dyadic $r_0\in(0,1]$ and define $I_0=[-r_0,r_0]$ as in~\eqref{e:I-0-def}.
For $n\geq 0$, define the set of all incidence pairs
\begin{equation}
  \label{e:incidence-set-def}
\mathscr I_n:=\{(F,\mathbf p)\in C^2(I_0)\times \mathcal D_r(\mathbb R^2)\colon
\Gamma_F\cap \mathbf p(nr)\neq\emptyset\}.
\end{equation}
Here $\mathbf p(nr)\subset\mathbb R^2$ is the $nr$-neighborhood of~$\mathbf p$, defined in~\eqref{e:X-r-def}.

We give here a slight modification of Theorem~\ref{t:OWY} which is better suited to our use:
%%%%%%%%%%%%%%%%%%%%%%%%%%%%%%%%%%%%%%%%%%%%%%%%%%%%%%%%%%%%%%%%%%%%%%%%%%%%%%%%
\begin{corollary}
  \label{c:OWY}
Let $\delta\in (0,1]$, $C_{\RKT},C_{\KT},\mathfrak T\geq 1$, $\epsilon>0$.
Then for all dyadic $r$ which are small enough depending only on $\mathfrak T,\delta,\epsilon,r_0$
the following holds. Let $\mathbb F\subset B_{C^2(I_0)}(0,1)$ be an $r$-separated
$\mathfrak T$-transversal family on $I_0$. Let $\mathbb P\subset \mathcal D_r(I'_0\times \mathbb R)$
be a family of dyadic $r$-cubes (recalling the notation~\eqref{e:D-r-X-def}
and~\eqref{e:I-0-def}).
Assume that:
\begin{itemize}
\item $\mathbb F$ satisfies the following RKT bound for all $x\in I_0$, all $r_1,r_2\in [r,1]$,
and all intervals $I_1,I_2\subset\mathbb R$ with $|I_j|=r_j$:
\begin{equation}
  \label{e:RKT-needed}
\#\{ F\in\mathbb F\colon A_x(F)\in I_1\times I_2\}\leq C_{\RKT}\Big({r_1r_2\over r^2}\Big)^\delta;
\end{equation}
\item for each $F\in\mathbb F$, the set $\{\mathbf p\in\mathbb P\colon (F,\mathbf p)\in \mathscr I_{8}\}$
is an $(r,\sigma,C_{\KT})$-KT cube family, where $\sigma:=2\min (\delta,1-\delta).$
\end{itemize}
Then
\begin{equation}
  \label{e:OWY-c}
\#((\mathbb F\times\mathbb P)\cap \mathscr I_8)
\leq C_{\RKT}^{\frac23}C_{\KT}^{\frac13}r^{-\frac23\delta-\epsilon}
\#(\mathbb P)^{\frac23}\#(\mathbb F)^{\frac13}.
\end{equation}
\end{corollary}
%%%%%%%%%%%%%%%%%%%%%%%%%%%%%%%%%%%%%%%%%%%%%%%%%%%%%%%%%%%%%%%%%%%%%%%%%%%%%%%%
\begin{proof}
In this proof, $C_0$ denotes a constant depending only on~$r_0$.

1. We rescale the functions in $\mathbb F$ to be defined on the interval $[-2,2]$.
Define the linear operator
$$
T_0:C^2(I_0)\to C^2([-2,2]),\quad
T_0 F(x)=F(\tfrac12 r_0x).
$$
Note that
\begin{equation}
  \label{e:OWY-c-int-1}
\tfrac14r_0^2\|F\|_{C^2(I_0)}\leq \|T_0F\|_{C^2([-2,2])}\leq \|F\|_{C^2(I_0)}.
\end{equation}
Thus the family
$$
\widehat {\mathbb F}:= T_0(\mathbb F)\ \subset\ B_{C^2([-2,2])}(0,1)
$$
is $\hat r$-separated where
$$
\hat r:=\tfrac14 r_0^2r.
$$
We have for all $x\in [-2,2]$
\begin{equation}
\label{e:OWY-c-int-2}
A_x(T_0 F)=Z(A_{\frac12 r_0x}(F))\quad\text{where }
Z(a,b)=(a,\tfrac12r_0b).
\end{equation}
Thus for each $F,G\in\mathbb F$ we have
$$
\begin{aligned}
\inf_{x\in [-2,2]}|A_x(T_0F)-A_x(T_0G)|
&\geq
\tfrac12r_0 \inf_{x\in I_0}|A_x(F)-A_x(G)|
\\
&\geq \tfrac12r_0 \mathfrak T^{-1}\|F-G\|_{C^2(I_0)}
\\
&\geq \tfrac12r_0 \mathfrak T^{-1}\|T_0F-T_0G\|_{C^2([-2,2])}
\end{aligned}
$$
where in the second inequality we use the transversality of $\mathbb F$. It follows
that $\widehat{\mathbb F}$ is $\widehat{\mathfrak T}$-transversal over $[-2,2]$,
with $\widehat{\mathfrak T}:=2r_0^{-1}\mathfrak T$.

\noindent 2. Assume that $I_1,I_2\subset\mathbb R$ are intervals with
$|I_j|=r_j$. We estimate for each $x\in [-2,2]$
$$
\begin{aligned}
\big|A_x(\widehat{\mathbb F})\cap (I_1\times I_2)\big|_{\hat r}
&\leq  \#\{F\in\mathbb F\colon A_x(T_0F)\in I_1\times I_2\}
\\
&=\#\{F\in\mathbb F\colon A_{\frac12r_0x}(F)\in Z^{-1}(I_1\times I_2)\}.
\end{aligned}
$$
We have $Z^{-1}(I_1\times I_2)=I_1\times \hat I_2$ where $\hat I_2$ is an interval
of length $2r_0^{-1}r_2$. Thus by~\eqref{e:RKT-needed} we see that
\begin{equation}
\label{e:OWY-c-int-3}
\big|A_x(\widehat{\mathbb F})\cap (I_1\times I_2)\big|_{\hat r}\leq 2r_0^{-1}C_{\RKT}\Big({r_1r_2\over r^2}\Big)^\delta\quad\text{if }r_1\in [r,1],\
r_2\in [\tfrac12 r_0r,\tfrac12 r_0].
\end{equation}
This estimate is easy to extend to the case when $r_1,r_2\in [\hat r,1]$ with
a loss of $C_0$ in the constant by covering any $r_1\times r_2$-rectangle by rectangles
of sizes allowed in~\eqref{e:OWY-c-int-3}.
It follows that $\widehat{\mathbb F}$ is an $(\hat r,2\delta,C_1)$-RKT set
with $C_1:=C_0C_{\RKT}$.

\noindent 3. Let $F\in\mathbb F$ and denote
$$
\mathbb P(F):=\{\mathbf p\in\mathbb P\colon (F,\mathbf p)\in\mathscr I_8\}.
$$
Assume that $\mathbf p\in \mathbb P(F)$.
Take $x\in I_0$ such that $(x,F(x))\in \mathbf p(8r)$. Then
$$
(2x/r_0,F(x))\in\Gamma_{T_0F}\cap S(\mathbf p(8r))\quad\text{where }
S(a,b):=(2a/r_0,b).
$$
Since $\mathbf p\cap (I_0'\times\mathbb R)\neq\emptyset$, by~\eqref{e:I-0-def} we see that
(for $r$ small enough) $x\in [-r_0/2,r_0/2]$. We also have $F(x)\in [-1,1]$.
Take
$$
\hat{\mathbf p}\in \mathcal D_{\hat r}([-1,1]^2)\quad\text{such that }(2x/r_0,F(x))\in\hat{\mathbf p}.
$$
This defines a map
$$
\Pi_F:\mathbb P(F)\to \mathcal D_{\hat r}([-1,1]^2),\quad
\Pi_F(\mathbf p)=\hat{\mathbf p},\quad
\Pi_F(\mathbf p)\cap\Gamma_{T_0F} \cap S(\mathbf p(8r))\neq\emptyset.
$$
Denote $\widehat F:=T_0F\in\widehat{\mathbb F}$ and define the cube family
$$
\widehat{\mathbb P}(\widehat F):=\Pi_F(\mathbb P(F)).
$$
Assume that $Q\subset \mathbb R^2$ is an $r'$-square with $r'\in [\hat r,1]$.
If $\mathbf p\in \mathbb P(F)$ and $\Pi_F(\mathbf p)\cap Q\neq\emptyset$,
then $\mathbf p\cap S^{-1}(Q(\hat r))(8r)\neq\emptyset$. The set
$S^{-1}(Q(\hat r))(8r)$ is contained in an $r'+16r+2\hat r$-square.
Since $\mathbb P(F)$ is an $(r,\sigma,C_{\KT})$-KT cube family, we see
that
$$
\#\{\mathbf p\in\mathbb P(F)\colon \Pi_F(\mathbf p)\cap Q\neq\emptyset\}
\leq C_0C_{\KT}\Big({r'\over \hat r}\Big)^\sigma.
$$
Thus $\widehat{\mathbb P}(\widehat F)$ is an $(\hat r,\sigma,C_2)$-KT cube family, where $C_2=C_0 C_{\KT}$.

\noindent 4. Put
$$
\widehat{\mathbb P}:=\bigcup_{\widehat F\in \widehat{\mathbb F}} \widehat{\mathbb P}(\widehat F).
$$
Assume that $\hat{\mathbf p}\in\widehat{\mathbb P}$. Then there exist $F\in \mathbb F$ and $\mathbf p\in \mathbb P(F)\subset\mathbb P$ such that $\hat{\mathbf p}=\Pi_F(\mathbf p)$.
We have
\begin{equation}
  \label{e:intersector}
\hat{\mathbf p}\cap S(\mathbf p(8r))\neq\emptyset.
\end{equation}
For each given $\mathbf p$,
there are at most $C_0$ many cubes $\hat{\mathbf p}$ satisfying~\eqref{e:intersector}. It follows that
$$
\#(\widehat{\mathbb P})\leq C_0\#(\mathbb P).
$$
Also, for each given $\hat{\mathbf p}$, there are at most $C_0$ many cubes $\mathbf p$
satisfying~\eqref{e:intersector}. Thus for each $F\in\mathbb F$
$$
\#(\mathbb P(F))\leq C_0\#(\widehat{\mathbb P}(T_0 F)).
$$
Now we apply Theorem~\ref{t:OWY} with $s:=\delta$ and $\widehat{\mathfrak T},\hat r,\widehat{\mathbb F},
\widehat{\mathbb P}(\widehat F)$ replacing $\mathfrak T,r,\mathbb F,\mathbb P(F)$ and see that
$$
\begin{aligned}
\#((\mathbb F\times\mathbb P)\cap\mathscr I_8)
&=\sum_{F\in\mathbb F}\#(\mathbb P(F))
\leq 
C_0\sum_{\widehat F\in\widehat{\mathbb F}} \#(\widehat{\mathbb P}(\widehat F))
\\
&\leq C_0 C_{\RKT}^{\frac23}C_{\KT}^{\frac13}r^{-\frac23 \delta-\epsilon}
\#(\mathbb P)^{\frac23 }\#(\mathbb F)^{\frac13}.
\end{aligned}
$$
This gives~\eqref{e:OWY-c} (if we change $\epsilon$ and take $r$ small enough).
\end{proof}
%%%%%%%%%%%%%%%%%%%%%%%%%%%%%%%%%%%%%%%%%%%%%%%%%%%%%%%%%%%%%%%%%%%%%%%%%%%%%%%%

%%%%%%%%%%%%%%%%%%%%%%%%%%%%%%%%%%%%%%%%%%%%%%%%%%%%%%%%%%%%%%%%%%%%%%%%%%%%%%%%
\subsection{Ingredients of the proof}
\label{s:AE-ingredients}

We henceforth make the assumptions of Proposition~\ref{l:additive-bound}:
\begin{itemize}
\item $\mathcal U\subset\mathbb R^2$ is an open neighborhood of the origin;
\item $\Theta\in C^\infty(\mathcal U;\mathbb R)$ has the Taylor expansion~\eqref{e:phase-cooked};
\item $r_0>0$ is dyadic and small enough depending on~$\Theta$;
\item we use the notation~\eqref{e:I-0-def}
for $I_0,I'_0$;
\item $\delta\in (0,\tfrac23]$ and $\epsilon>0$;
\item $r\leq 1$ is a dyadic number;
\item $X,Y\subset I'_0$ are $\delta$-dimensional AD regular on scales~$r$ to~1 with constant $C_{\AD}$.
\end{itemize}
Denote by $C_\Theta$ a constant depending only on $\Theta$ and by
$C$ a constant depending only on $\Theta,r_0,\delta,\epsilon,C_{\AD}$.

%%%%%%%%%%%%%%%%%%%%%%%%%%%%%%%%%%%%%%%%%%%%%%%%%%%%%%%%%%%%%%%%%%%%%%%%%%%%%%%%
\subsubsection{Our family of functions}

Denote
$$
\tilde y=(y_1,y_2),\
\tilde y'=(y'_1,y'_2)\ \in\ \mathbb R^2.
$$
For $\tilde y\in I_0^2$, define
\begin{equation}
  \label{e:our-F-def}
F_{\tilde y}\in C^2(I_0),\quad
F_{\tilde y}(x):=\partial_x\Theta(x,y_1)+\partial_x\Theta(x,y_2).
\end{equation}
Then the quantity that we need to estimate in~\eqref{e:additive-bound}
can be written as follows:
\begin{equation}
  \label{e:AE-rewritten}
|(X\times Y^4)\cap \mathcal E_r(\Theta)|
=\big|\big\{
(x,\tilde y,\tilde y')\in X\times Y^4\colon
|F_{\tilde y}(x)-F_{\tilde y'}(x)|\leq r
\big\}\big|.
\end{equation}
For $x\in I_0$, define
\begin{equation}
  \label{e:A-x-cal-def}
\mathcal A_x(\tilde y):=A_x(F_{\tilde y})=(F_{\tilde y}(x),\partial_x F_{\tilde y}(x))
\ \in\ \mathbb R^2.
\end{equation}
See Figure~\ref{f:geometry-grid}.

%%%%%%%%%%%%%%%%%%%%%%%%%%%%%%%%%%%%%%%%%%%%%%%%%%%%%%%%%%%%%%%%%%%%%%%%%%%%%%%%
\begin{figure}
\includegraphics{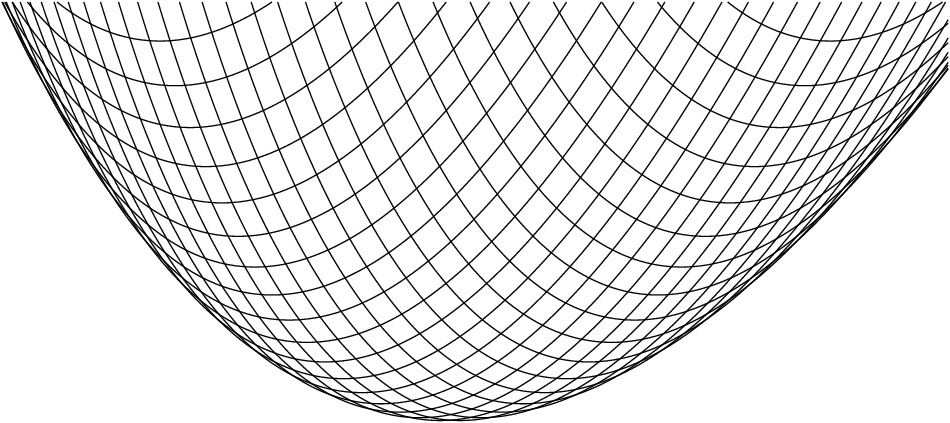}
\caption{The images of the grid lines $\{y_1=\const\}$ and
$\{y_2=\const\}$ under the map
$\tilde y\mapsto \mathcal A_x(\tilde y)$. Here
we take 
$\Theta(x,y)=xy+x^2y^2$ and $x=0.1$, so
$\mathcal A_x(y_1,y_2)=(y_1+y_2+0.2(y_1^2+y_2^2),2(y_1^2+y_2^2))$.
Note that the grid lines compress near the image of the diagonal $\{y_1=y_2\}$,
which is consistent with the degeneration of the Jacobian~\eqref{e:A-x-J}.}
\label{f:geometry-grid}
\end{figure}
%%%%%%%%%%%%%%%%%%%%%%%%%%%%%%%%%%%%%%%%%%%%%%%%%%%%%%%%%%%%%%%%%%%%%%%%%%%%%%%%

%%%%%%%%%%%%%%%%%%%%%%%%%%%%%%%%%%%%%%%%%%%%%%%%%%%%%%%%%%%%%%%%%%%%%%%%%%%%%%%%
\subsubsection{Transversality}

Since $F_{(y_1,y_2)}=F_{(y_2,y_1)}$, we see that $F_{\tilde y}(x)$ is a smooth function
of $(x,\hat y)$, where we put
\begin{equation}
  \label{e:y-pm-def}
\hat y:=(y_+,y_-),\quad
y_+:=y_1+y_2,\quad
y_-:=(y_1-y_2)^2.
\end{equation}
Note that $y_1^2+y_2^2=\tfrac12 (y_- + y_+^2)$ and
\begin{equation}
  \label{e:hat-y-det}
\partial_{\tilde y}\hat y=\begin{pmatrix}
1 & 1 \\
2(y_1-y_2) & 2(y_2-y_1)
\end{pmatrix}.
\end{equation}
The Taylor expansion~\eqref{e:phase-cooked} fixes all derivatives
$\partial^\alpha\Theta(0,0)$ with $|\alpha|\leq 4$, which gives
the Taylor expansions for the derivatives of~$\Theta$ as $(x,y)\to (0,0)$:
\begin{equation}
  \label{e:taylor-Theta}
\begin{aligned}
\partial_x\Theta(x,y) &=
y+2xy^2+\mathcal O(|(x,y)|^4),
\\
\partial_x^2\Theta(x,y) &=
2y^2 + \mathcal O(|(x,y)|^3).
\end{aligned}
\end{equation}
These in turn imply that
\begin{equation}
  \label{e:taylor-F}
\begin{aligned}
F_{\tilde y}(x) &=
y_1+y_2+2x(y_1^2+y_2^2)+\mathcal O(|(x,\tilde y)|^4),
\\
\partial_x F_{\tilde y}(x) &=
2(y_1^2+y_2^2) + \mathcal O(|(x,\tilde y)|^3),
\\
\partial_x^2 F_{\tilde y}(x) &=
\mathcal O(|(x,\tilde y)|^2).
\end{aligned}
\end{equation}
In particular, $\|F_{\tilde y}\|_{C^2(I_0)}\leq C_\Theta r_0$.

From the expansions~\eqref{e:taylor-F}, we get the following expansions
in the $(x,\hat y)$ coordinates:
\begin{equation}
  \label{e:taylor-F-new}
\begin{aligned}
F_{\tilde y}(x) &= y_+ + x(y_- + y_+^2) + \mathcal O(|x|^4+|y_-|^2+|y_+|^4),
\\
\partial_x F_{\tilde y}(x) &= y_- + y_+^2 + \mathcal O(|x|^3 + |y_-|^{\frac32}
+|y_+|^3).
\end{aligned}
\end{equation}
Moreover, the map $\mathcal A_x(\tilde y)$ defined in~\eqref{e:A-x-cal-def} depends smoothly on $x,\hat y$ and
\begin{equation}
  \label{e:A-x-d-hat}
\partial_{\hat y}\mathcal A_x(\tilde y) =
\begin{pmatrix}
1 & 0 \\ 0 & 1
\end{pmatrix}
+\mathcal O(|(x,\hat y)|).
\end{equation}
Combining this with~\eqref{e:hat-y-det}, we see that
the Jacobian of $\mathcal A_x$ in the $\tilde y$ variables is given by
\begin{equation}
  \label{e:A-x-J}
|\det \partial_{\tilde y}\mathcal A_x(\tilde y)|= 4|y_1-y_2|(1+\mathcal O(r_0))\quad\text{for all }
x\in I_0,\ \tilde y\in I_0^2.
\end{equation}
We can now show transversality of our family of functions:
%%%%%%%%%%%%%%%%%%%%%%%%%%%%%%%%%%%%%%%%%%%%%%%%%%%%%%%%%%%%%%%%%%%%%%%%%%%%%%%%
\begin{lemma}
\label{l:distance-verified}
Assume that $\tilde y=(y_1,y_2),\tilde y'=(y'_1,y'_2)\in I_0^2$ and define
$\hat y,\hat y'$ following~\eqref{e:y-pm-def}. Then for $r_0$ small enough
depending only on~$\Theta$
and some constant $\widehat C_\Theta$ depending only on~$\Theta$, we have
\begin{align}
  \label{e:distance-verified-0}
\|F_{\tilde y}-F_{\tilde y'}\|_{C^2(I_0)} &\leq 3 |\tilde y-\tilde y'|,
\\
  \label{e:distance-verified-1}
\|F_{\tilde y}-F_{\tilde y'}\|_{C^2(I_0)} &\leq \widehat C_\Theta|\hat y-\hat y'|,
\\  
  \label{e:distance-verified-2}
\inf_{x\in I_0}|\mathcal A_x(\tilde y)-\mathcal A_x(\tilde y')|
&\geq \tfrac12 |\hat y-\hat y'|.
\end{align}
Consequently the family $\{F_{\tilde y}\colon \tilde y\in I_0^2\}$
is $\mathfrak T_{\Theta}$-transversal on~$I_0$ for $\mathfrak T_\Theta:=2\widehat C_\Theta$.
\end{lemma}
%%%%%%%%%%%%%%%%%%%%%%%%%%%%%%%%%%%%%%%%%%%%%%%%%%%%%%%%%%%%%%%%%%%%%%%%%%%%%%%%
\begin{proof}
To see~\eqref{e:distance-verified-0}, we use~\eqref{e:taylor-F} to show that
for each $x\in I_0$
$$
\begin{gathered}
(F_{\tilde y}(x)-F_{\tilde y'}(x),
\partial_x F_{\tilde y}(x)-\partial_x F_{\tilde y'}(x),
\partial_x^2 F_{\tilde y}(x)-\partial_x^2 F_{\tilde y'}(x))
\\
=(y_1-y'_1+y_2-y'_2,0,0)+\mathcal O(r_0|\tilde y-\tilde y'|).
\end{gathered}
$$
The estimate~\eqref{e:distance-verified-1} is immediate because
$F_{\tilde y}(x)$ is smooth in $x,\hat y$. To show~\eqref{e:distance-verified-2},
we note that by~\eqref{e:A-x-d-hat} and
the smoothness of $\mathcal A_{\tilde y}(x)$ in $x,\hat y$ we have
for all $x\in I_0$
$$
\mathcal A_x(\tilde y)-\mathcal A_x(\tilde y')=\hat y-\hat y'+\mathcal O(r_0|\hat y-\hat y'|).
$$
It remains to make $r_0$ small enough so that $\mathcal O(r_0)\leq \frac12$.
\end{proof}
%%%%%%%%%%%%%%%%%%%%%%%%%%%%%%%%%%%%%%%%%%%%%%%%%%%%%%%%%%%%%%%%%%%%%%%%%%%%%%%%

%%%%%%%%%%%%%%%%%%%%%%%%%%%%%%%%%%%%%%%%%%%%%%%%%%%%%%%%%%%%%%%%%%%%%%%%%%%%%%%%
\subsubsection{Discretization and pigeonholing}

The Jacobian formula~\eqref{e:A-x-J} suggests that we should split the space
of possible $\tilde y$ into dyadic pieces according to the size of $|y_1-y_2|$. Define
the set of scales
$$
\widetilde{\mathcal R}:= \{\tilde r\text{ dyadic}\colon \sqrt r\leq \tilde r\leq 2r_0\}\cup\{0\}.
$$
For $\tilde r\in\widetilde{\mathcal R}$, put
\begin{equation}
\label{e:Sigma-r-def}
\Sigma_{\tilde r}:=\begin{cases}
\{(y_1,y_2)\in I_0^2\colon \tilde r\leq |y_1-y_2|\leq 2\tilde r\}, &\text{if }\tilde r >0;
\\
\{(y_1,y_2)\in I_0^2\colon |y_1-y_2|\leq 2\sqrt r\}, &\text{if }\tilde r=0.
\end{cases}
\end{equation}
Then
\begin{equation}
  \label{e:Sigma-union}
I_0^2=\bigcup_{\tilde r\in\widetilde{\mathcal R}}\Sigma_{\tilde r}.
\end{equation}
The discretization of $Y^2\cap \Sigma_{\tilde r}$ is given by the set
$\mathcal D_r(Y^2\cap \Sigma_{\tilde r})$ of dyadic $r$-squares intersecting it
(see~\eqref{e:D-r-X-def}).
The next lemma discretizes the corresponding set of functions~$F_{\tilde y}$:
%%%%%%%%%%%%%%%%%%%%%%%%%%%%%%%%%%%%%%%%%%%%%%%%%%%%%%%%%%%%%%%%%%%%%%%%%%%%%%%%
\begin{lemma}
  \label{l:function-projection}
For each $\tilde r\in\widetilde{\mathcal R}$, there exists an $r$-separated
set of functions in $C^2(I_0)$
\begin{equation}
  \label{e:F-r-tilde-def}
\mathbb F_{\tilde r} \subset \{F_{\tilde y}\colon \tilde y\in Y^2\cap \Sigma_{\tilde r}\}
\end{equation}
and a surjective map
\begin{equation}
  \label{e:function-projection}
\tilde{\mathbf y}\in\mathcal D_r(Y^2\cap\Sigma_{\tilde r})\ \mapsto\
F_{\tilde{\mathbf y},\tilde r}\in \mathbb F_{\tilde r}
\end{equation}
such that for each $\tilde{\mathbf y}\in\mathcal D_r(Y^2\cap\Sigma_{\tilde r})$
we have
\begin{equation}
  \label{e:function-projection-close}
\|F_{\tilde y}-F_{\tilde{\mathbf y},\tilde r}\|_{C^2(I_0)}\leq 7 r\quad\text{for all }\tilde y\in\tilde{\mathbf y}.
\end{equation}
\end{lemma}
%%%%%%%%%%%%%%%%%%%%%%%%%%%%%%%%%%%%%%%%%%%%%%%%%%%%%%%%%%%%%%%%%%%%%%%%%%%%%%%%
\begin{proof}
Let $\mathbb F_{\tilde r}$ be a maximal $r$-separated set in
$\{F_{\tilde y}\colon \tilde y\in Y^2\cap \Sigma_{\tilde r}\}$
with respect to the $C^2(I_0)$ norm. For each $\tilde{\mathbf y}\in\mathcal D_r(Y^2\cap \Sigma_{\tilde r})$,
pick $\tilde y_{\tilde{\mathbf y}}\in \tilde{\mathbf y}\cap Y^2\cap \Sigma_{\tilde r}$.
By the maximality of~$\mathbb F_{\tilde r}$, we can choose
$F_{\tilde{\mathbf y},\tilde r}\in\mathbb F_{\tilde r}$ such that
$$
\|F_{\tilde y_{\tilde{\mathbf y}}}-F_{\tilde{\mathbf y},\tilde r}\|_{C^2(I_0)}\leq r.
$$
For each $\tilde y\in \tilde{\mathbf y}$,
we have $|\tilde y-\tilde y_{\tilde{\mathbf y}}|\leq 2r$ and thus by~\eqref{e:distance-verified-0}
$$
\|F_{\tilde y}-F_{\tilde y_{\tilde{\mathbf y}}}\|_{C^2(I_0)}\leq 6r
$$
This gives~\eqref{e:function-projection-close}.
Finally, we make the map~\eqref{e:function-projection} surjective
by shrinking $\mathbb F_{\tilde r}$ if necessary. 
\end{proof}
%%%%%%%%%%%%%%%%%%%%%%%%%%%%%%%%%%%%%%%%%%%%%%%%%%%%%%%%%%%%%%%%%%%%%%%%%%%%%%%%
When $\tilde r$ is small, many different $r$-squares $\tilde{\mathbf y}$ can
project to the same function $F_{\tilde{\mathbf y},\tilde r}$.
We thus define the multiplicity of a function: for $F\in \mathbb F_{\tilde r}$, put
\begin{equation}
  \label{e:multiplicity-def}
M_{\tilde r}(F):=\#\{\tilde{\mathbf y}\in \mathcal D_r(Y^2\cap \Sigma_{\tilde r})\colon
F_{\tilde{\mathbf y},\tilde r}=F
\}.
\end{equation}
We now split the set of functions according to multiplicity.
For a dyadic number $m\geq 1$, define
\begin{equation}
  \label{e:F-r-m-def}
\mathbb F_{\tilde r,m}:=\{F\in \mathbb F_{\tilde r}\colon m\leq M_{\tilde r}(F)<2m\},\quad
\mathbb F_{\tilde r}=\bigsqcup_{m\geq 1} \mathbb F_{\tilde r,m}.
\end{equation}
The set $\mathbb F_{\tilde r,m}$ is $\mathfrak T_{\Theta}$-transversal by Lemma~\ref{l:distance-verified}
and $r$-separated by Lemma~\ref{l:function-projection}.

We finish this subsection with an upper bound on $\#(\mathbb F_{\tilde r,m})$.
Assume that $\tilde{\mathbf y}=\mathbf y_1\times\mathbf y_2\in \mathcal D_r(Y^2\cap\Sigma_{\tilde r})$.
Then $\mathbf y_1\in \mathcal D_r(Y)$ and $\mathbf y_2\in \mathcal D_r(Y\cap \mathbf y_1(2(\tilde r+\sqrt r)))$. 
The interval $I:=\mathbf y_1(2(\tilde r+\sqrt r))$ satisfies
$|I|\leq 5(\tilde r+\sqrt r)$. We have for each such interval~$I$
\begin{equation}
  \label{e:Y-discrete-count}
\begin{aligned}
\#(\mathcal D_r(Y)) &\leq 3|Y|_r\leq 36C_{\AD}^2 r^{-\delta},
\\
\#(\mathcal D_r(Y\cap I)) &\leq 3|Y\cap I|_r \leq 180 C_{\AD}^2(\tilde r+\sqrt r)^\delta r^{-\delta}.
\end{aligned}
\end{equation}
Here in the first inequalities we use~\eqref{e:D-r-bound} and in the second inequalities
we use Lemma~\ref{l:AD-cover}. This implies that
$$
\sum_{F\in\mathbb F_{\tilde r}}M_{\tilde r}(F)
=\#(\mathcal D_r(Y^2\cap\Sigma_{\tilde r}))\leq C(\tilde r+\sqrt r)^\delta r^{-2\delta}
$$
which in turn gives
\begin{equation}
  \label{e:F-r-m-bound}
\#(\mathbb F_{\tilde r,m})\leq Cm^{-1}(\tilde r+\sqrt r)^\delta r^{-2\delta}.
\end{equation}

%%%%%%%%%%%%%%%%%%%%%%%%%%%%%%%%%%%%%%%%%%%%%%%%%%%%%%%%%%%%%%%%%%%%%%%%%%%%%%%%
\subsubsection{Geometry of lines and circles}
\label{s:geometry}

In preparation for the proof of the rectangular Katz--Tao property of the set $\mathbb F_{\tilde r,m}$
(see Lemma~\ref{l:RKT-bound} below), we will study the level curves of the components of
$$
\mathcal A_x(\tilde y)=(F_{\tilde y}(x),\partial_x F_{\tilde y}(x))=
(\partial_x\Theta(x,y_1)+\partial_x\Theta(x,y_2),
\partial_x^2 \Theta(x,y_1)+\partial_x^2 \Theta(x,y_2)),
$$
together with the level sets of the function $y_1-y_2$. See Figure~\ref{f:geometry}.

The next lemma establishes the basic properties of the components
of the functions $\partial_x\Theta(x,y)$ and $\partial_x^2\Theta(x,y)$.
%%%%%%%%%%%%%%%%%%%%%%%%%%%%%%%%%%%%%%%%%%%%%%%%%%%%%%%%%%%%%%%%%%%%%%%%%%%%%%%%
\begin{lemma}
\label{l:geometry-fns}
For $r_0$ small enough depending only on~$\Theta$,
we have on $[-10r_0,10r_0]^2$
$$
\begin{aligned}
\partial_x \Theta(x,y) &= a_x(y),
\\
\partial_x^2 \Theta(x,y) &= c_{\min}(x) + 2b_x(y)^2
\end{aligned}
$$
where the functions $a_x(y),b_x(y)$ are smooth in $(x,y)\in [-10r_0,10r_0]^2$,
the function $c_{\min}(x)$ is smooth in $x\in [-10r_0,10r_0]$, and
\begin{gather}
  \label{e:geofun-a}
a_x(y)=y+\mathcal O(r_0^3),\quad
\partial_y a_x(y)=1+\mathcal O(r_0^2),\quad
\partial_y^2 a_x(y)=\mathcal O(r_0);
\\
  \label{e:geofun-b}
b_x(y)=y+\mathcal O(r_0^2),\quad
\partial_y b_x(y)=1+\mathcal O(r_0);
\\
  \label{e:geofun-c}
c_{\min}(x) = \mathcal O(r_0^3).
\end{gather}
Moreover, the equation $b_x(y)=0$ has a unique solution $y=y_{\min}(x)$ on $[-10r_0,10r_0]$,
which depends smoothly on~$x$ and satisfies
\begin{equation}
  \label{e:geofun-min}
y_{\min}(x)=\mathcal O(r_0^2).
\end{equation}
\end{lemma}
%%%%%%%%%%%%%%%%%%%%%%%%%%%%%%%%%%%%%%%%%%%%%%%%%%%%%%%%%%%%%%%%%%%%%%%%%%%%%%%%
\begin{remark}
\label{r:useful-circle}
It might be useful
to keep in mind the model case when $\Theta(x,y)=xy+x^2y^2$.
Then $a_x(y)=y+2xy^2$, $b_x(y)=y$, $c_{\min}(x)=0$, and $y_{\min}(x)=0$.
\end{remark}
%%%%%%%%%%%%%%%%%%%%%%%%%%%%%%%%%%%%%%%%%%%%%%%%%%%%%%%%%%%%%%%%%%%%%%%%%%%%%%%%
%%%%%%%%%%%%%%%%%%%%%%%%%%%%%%%%%%%%%%%%%%%%%%%%%%%%%%%%%%%%%%%%%%%%%%%%%%%%%%%%
\begin{figure}
\includegraphics{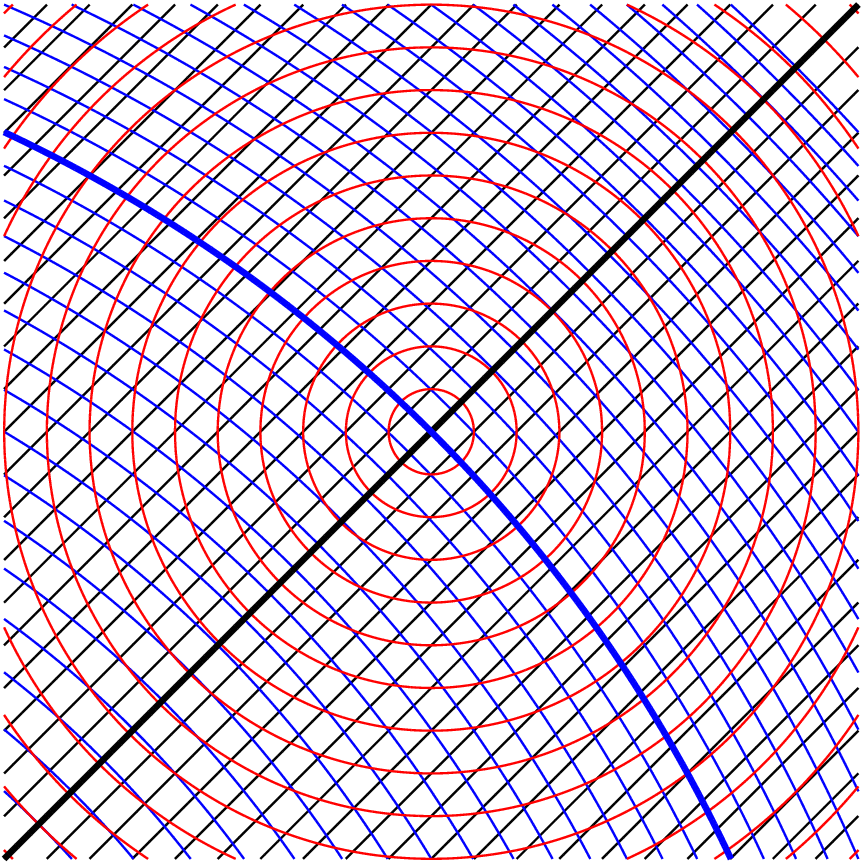}
\caption{The level curves in the $(y_1,y_2)$ plane of the first component of $\mathcal A_x(\tilde y)$
(in blue), the second component of $\mathcal A_x(\tilde y)$ (in red), and $y_1-y_2$ (in black).
Here $\Theta,x$ are the same as in Figure~\ref{f:geometry-grid}. The thick lines
separate the plane into the quadrants defined in~\eqref{e:quadrants}.}
\label{f:geometry}
\end{figure}
%%%%%%%%%%%%%%%%%%%%%%%%%%%%%%%%%%%%%%%%%%%%%%%%%%%%%%%%%%%%%%%%%%%%%%%%%%%%%%%%

\begin{proof}
From~\eqref{e:taylor-Theta} we have
\begin{equation}
  \label{e:taylor-Theta-2}
\begin{aligned}
\partial_x\Theta(x,y) &= y + 2xy^2+ \mathcal O(|(x,y)|^4),
\\
\partial_y\partial_x \Theta(x,y) &= 1 + 4xy + \mathcal O(|(x,y)|^3),
\\
\partial_y^2\partial_x \Theta(x,y) &= 4x + \mathcal O(|(x,y)|^2),
\\
\partial_x^2\Theta(x,y) &= 2y^2 + \mathcal O(|(x,y)|^3),
\\
\partial_y \partial_x^2 \Theta(x,y) &= 4y + \mathcal O(|(x,y)|^2),
\\
\partial_y^2\partial_x^2 \Theta(x,y) &= 4 + \mathcal O(|(x,y)|).
\end{aligned}
\end{equation}
The properties~\eqref{e:geofun-a} of $a_x(y)$ now follow immediately.
Next, we see that the function $y\mapsto \partial_y\partial_x^2\Theta(x,y)$
is strictly increasing on $[-10r_0,10r_0]$ and has opposite signs
on the endpoints of this interval. It follows that
the equation
$$
\partial_y \partial_x^2\Theta(x,y)=0
$$
has a unique solution $y\in [-10r_0,10r_0]$, which we denote
by $y_{\min}(x)$. This function is smooth in $x\in [-10r_0,10r_0]$ and~\eqref{e:geofun-min}
holds. We have $\partial_y^2\partial_x^2\Theta>0$, so
the function $y\mapsto \partial_x^2\Theta(x,y)$ is strictly convex
and $y_{\min}(x)$ is its minimum point. Define
$$
c_{\min}(x):=\partial_x^2\Theta(x,y_{\min}(x)),
$$
then $c_{\min}(x)$ depends smoothly on $x\in [-10r_0,10r_0]$
and~\eqref{e:geofun-c} holds.

Finally, using the Taylor expansion at $y=y_{\min}(x)$ we have
$$
\begin{aligned}
\partial_x^2\Theta(x,y)&=c_{\min}(x)
+2(y-y_{\min}(x))^2 B(x,y),
\\
B(x,y)\,&:=\frac12 \int_0^1 (1-t)\partial_y^2\partial_x^2\Theta(x,ty+(1-t)y_{\min}(x))\,dt.
\end{aligned}
$$
The function $B(x,y)$ is smooth on $[-10r_0,10r_0]^2$ and satisfies $B(x,y)=1+\mathcal O(r_0)$. We now put
$$
b_x(y):=(y-y_{\min}(x))\sqrt{B(x,y)}
$$
and this function satisfies~\eqref{e:geofun-b}.
\end{proof}
%%%%%%%%%%%%%%%%%%%%%%%%%%%%%%%%%%%%%%%%%%%%%%%%%%%%%%%%%%%%%%%%%%%%%%%%%%%%%%%%
In terms of the functions introduced in Lemma~\ref{l:geometry-fns}, we have
\begin{equation}
  \label{e:geometry-fns-A}
\mathcal A_x(y_1,y_2)=\big(a_x(y_1)+a_x(y_2),2(c_{\min}(x)+b_x(y_1)^2+b_x(y_2)^2)\big).
\end{equation}
The function $\tilde y\mapsto \partial_xF_{\tilde y}(x)$ has a local minimum
at $\tilde y=(y_{\min}(x),y_{\min}(x))$. We split $[-10r_0,10r_0]^2$ into four quadrants
using the level curves of $F_{\tilde y}(x)$ and $y_1-y_2$ passing through that point:
\begin{equation}
  \label{e:quadrants}
\begin{gathered}
\mathcal Q_{\varepsilon_1\varepsilon_2}(x)\,
:=\{\tilde y\in [-10r_0,10r_0]^2\colon
\varepsilon_1(F_{\tilde y}(x)-2a_{\min}(x))\geq 0,\
\varepsilon_2(y_2-y_1)\geq 0\}
\\
\text{where }\varepsilon_1,\varepsilon_2 \in\{+,-\},\quad
a_{\min}(x):=a_x(y_{\min}(x)).
\end{gathered}
\end{equation}
We will in particular focus on the top quadrant
\begin{equation}
  \label{e:quadrant++}
\mathcal Q_{++}(x):=
\{\tilde y\in [-10r_0,10r_0]^2\colon
F_{\tilde y}(x)\geq 2a_{\min}(x),\
y_2\geq y_1\}.
\end{equation}
We establish the following technical lemma (which again is easiest to verify in the case
described in Remark~\ref{r:useful-circle}):
%%%%%%%%%%%%%%%%%%%%%%%%%%%%%%%%%%%%%%%%%%%%%%%%%%%%%%%%%%%%%%%%%%%%%%%%%%%%%%%%
\begin{lemma}
  \label{l:quadrant-circle}
Let $x,y_1,y_2\in [-10r_0,10r_0]$, and $\eta=b_x(y_1)^2+b_x(y_2)^2$. Then:
\begin{enumerate}
\item if $(y_1,y_2)\in\mathcal Q_{++}(x)$, then $y_2\geq y_{\min}(x)$ and
$|y_1-y_{\min}(x)|\leq 0.8\sqrt\eta$.
\item if $y_2\geq y_{\min}(x)$ and $|y_1-y_{\min}(x)|\leq 0.8\sqrt\eta$,
then $b_x(y_2)\geq \tfrac12 \sqrt\eta$ and $y_2-y_{\min}(x)\geq \tfrac12\sqrt\eta$.
\end{enumerate}
\end{lemma}
%%%%%%%%%%%%%%%%%%%%%%%%%%%%%%%%%%%%%%%%%%%%%%%%%%%%%%%%%%%%%%%%%%%%%%%%%%%%%%%%
\begin{proof}
(1) Since $a_x$ is increasing and $(y_1,y_2)\in\mathcal Q_{++}(x)$, we have
$$
|a_x(y_1)-a_{\min}(x)|\leq a_x(y_2)-a_{\min}(x).
$$
In particular, $y_2\geq y_{\min}(x)$ and $|y_1-y_{\min}(x)|\leq (1+\mathcal O(r_0^2))(y_2-y_{\min}(x))$.
This in turn implies that $|b_x(y_1)|\leq (1+\mathcal O(r_0))b_x(y_2)$.
Since $\eta=b_x(y_1)^2+b_x(y_2)^2$, we see that $b_x(y_1)^2\leq (\tfrac12+\mathcal O(r_0))\eta$
and thus $|y_1-y_{\min}(x)|\leq 0.8\sqrt\eta$.

\noindent (2) We have $|b_x(y_1)|\leq (0.8+\mathcal O(r_0))\sqrt\eta$.
Since $\eta=b_x(y_1)^2+b_x(y_2)^2$, we see that $b_x(y_2)^2\geq (0.36-\mathcal O(r_0))\eta$,
which implies that $b_x(y_2)\geq \tfrac12\sqrt\eta$
and $y_2-y_{\min}(x)\geq \tfrac12\sqrt\eta$.
\end{proof}
%%%%%%%%%%%%%%%%%%%%%%%%%%%%%%%%%%%%%%%%%%%%%%%%%%%%%%%%%%%%%%%%%%%%%%%%%%%%%%%%
Next, define the inverses of the maps $a_x,b_x$:
$$
a_x^{-1}:a_x([-10r_0,10r_0])\to [-10r_0,10r_0],\quad
b_x^{-1}:b_x([-10r_0,10r_0])\to [-10r_0,10r_0].
$$
These exist because $\partial_y a_x,\partial_y b_x>0$ and we have
$$
[-9r_0,9r_0]\ \subset\ a_x([-10r_0,10r_0])\cap b_x([-10r_0,10r_0]).
$$
We also see from~\eqref{e:geofun-a} and~\eqref{e:geofun-b} that
for all $z\in [-9r_0,9r_0]$
\begin{gather}
  \label{e:geofun-inv-a}
a_x^{-1}(z)=z+\mathcal O(r_0^3),\quad
\partial_z a_x^{-1}(z)=1+\mathcal O(r_0^2),\quad
\partial_z^2 a_x^{-1}(z)=\mathcal O(r_0);
\\
  \label{e:geofun-inv-b}
b_x^{-1}(z)=z+\mathcal O(r_0^2),\quad
\partial_z b_x^{-1}(z)=1+\mathcal O(r_0).
\end{gather}
We are now ready to describe the intersections of the level curves
of the components of $\mathcal A_x(\tilde y)$ with $\mathcal Q_{++}(x)$
as graphs of functions:
%%%%%%%%%%%%%%%%%%%%%%%%%%%%%%%%%%%%%%%%%%%%%%%%%%%%%%%%%%%%%%%%%%%%%%%%%%%%%%%%
\begin{definition}
  \label{d:admissible}
Let $J\subset I_0$ be an interval. Fix $x\in I_0$. Define the set
of \emph{admissible functions} as functions $f:J\to \mathbb R$ of one of the following forms:
\begin{enumerate}
\item[(i)] $f(y)=y+\eta$ where $\eta\in [-5r_0,5r_0]$;
\item[(ii)] $f(y)=a_x^{-1}(\eta-a_x(y))$ where $\eta\in [-5r_0,5r_0]$;
\item[(iii)] $f(y)=b_x^{-1}(\sqrt{\eta-b_x(y)^2})$
where $\eta\in (0,5r_0^2]$ and we additionally impose the restriction that
$J\subset [y_{\min}(x)-0.8\sqrt\eta,y_{\min}(x)+0.8\sqrt\eta]$. 
\end{enumerate}
\end{definition}
%%%%%%%%%%%%%%%%%%%%%%%%%%%%%%%%%%%%%%%%%%%%%%%%%%%%%%%%%%%%%%%%%%%%%%%%%%%%%%%%
Note that the graph $\{(y_1,y_2)\colon y_1\in J,\ y_2=f(y_1)\}$ of a function of type~(i) lies on a level curve
of $y_1-y_2$, the graph of a function of type~(ii) lies on a level curve of $\tilde y\mapsto F_{\tilde y}(x)$,
and the graph of a function of type~(iii) lies on a level curve of $\tilde y\mapsto \partial_x F_{\tilde y}(x)$.

We also note that in the model case considered in Remark~\ref{r:useful-circle}, 
if we further set $x:=0$, functions
of type~(ii) have the form $f(y)=\eta-y$ and functions of type~(iii)
have the form $f(y)=\sqrt{\eta-y^2}$ where $|y|\leq 0.8\sqrt\eta$.

The following are the key analytic properties of admissible functions,
where we use the definitions from~\S\ref{s:moderate}:
%%%%%%%%%%%%%%%%%%%%%%%%%%%%%%%%%%%%%%%%%%%%%%%%%%%%%%%%%%%%%%%%%%%%%%%%%%%%%%%%
\begin{lemma}
  \label{l:lines-circles}
Fix $x\in I_0$. Then:
\begin{enumerate}
\item Let $f$ be an admissible function on~$J$. Then $f\in\Lip_{10}(J)$.
\item Let $f,g$ be two different admissible functions on~$J$. Then:
\begin{enumerate}
\item If $f,g$ belong to the same type (i), (ii), or~(iii) above, then $f,g$ are parallel.
\item If $f,g$ belong to two different types, then $f,g$ are transversal.
\end{enumerate}
\end{enumerate}
\end{lemma}
%%%%%%%%%%%%%%%%%%%%%%%%%%%%%%%%%%%%%%%%%%%%%%%%%%%%%%%%%%%%%%%%%%%%%%%%%%%%%%%%
\begin{proof}
We first study the first and second derivatives of admissible functions:
\begin{enumerate}
\item[(i)] We have $f'(y)=1$ and $f''(y)=0$.
\item[(ii)] By~\eqref{e:geofun-a} and~\eqref{e:geofun-inv-a} we have $f'(y)=-1+\mathcal O(r_0^2)$ and $f''(y)=\mathcal O(r_0)$.
\item[(iii)] We have $|f'(y)|\leq 3$ and $f''(y)\leq -1/(2\sqrt\eta)$.
To see this we use~\eqref{e:geofun-b} and differentiate the identity
$b_x(y)^2+b_x(f(y))^2=\eta$ in~$y$. This gives
$$
\begin{aligned}
f'(y) &= -(1+\mathcal O(r_0)){b_x(y)\over b_x(f(y))},
\\
f''(y) &= -{(1+\mathcal O(r_0))(1+f'(y)^2) \over b_x(f(y))}.
\end{aligned}
$$
It remains to use Lemma~\ref{l:quadrant-circle}(2) and the fact
that $|y-y_{\min}(x)|\leq 0.8\sqrt\eta$ which together give
$b_x(f(y))\geq \frac12\sqrt\eta$.
\end{enumerate}

We are now ready to prove the statements in the lemma:

\noindent\underline{(1):} it follows from the above computations that $|f'(y)|\leq 3$.

\smallskip

\noindent\underline{(2)(a):} We consider the following cases. In each case $\eta_f,\eta_g$ are
the values of the parameter $\eta$ used to define $f,g$.
\begin{itemize}
\item $f,g$ both type~(i): we have $|g(y)-f(y)|=|\eta_f-\eta_g|$ which is a positive constant.
\item $f,g$ both type~(ii):
we have $\tfrac12|\eta_f-\eta_g|\leq |g(y)-f(y)|\leq 2|\eta_f-\eta_g|$
for all $y\in J$.
\item $f,g$ both type~(iii): Assume that $\eta_f<\eta_g$ and take $y$ such that
$|y-y_{\min}(x)|\leq 0.8\sqrt{\eta_f}$. We have
$$
b_x(y)^2+b_x(f(y))^2=\eta_f,\quad
b_x(y)^2+b_x(g(y))^2=\eta_g.
$$
By Lemma~\ref{l:quadrant-circle}(2), we get
$\tfrac12\sqrt{\eta_f} \leq b_x(f(y))\leq \sqrt{\eta_f}$
and $\tfrac12\sqrt{\eta_g}\leq b_x(g(y))\leq \sqrt{\eta_g}$.
Now
$$
b_x(g(y))-b_x(f(y))={\eta_g-\eta_f\over b_x(g(y))+b_x(f(y))}
\ \in\
[\sqrt{\eta_g}-\sqrt{\eta_f},2(\sqrt{\eta_g}-\sqrt{\eta_f})]
$$
which implies that
$$
g(y)-f(y)\in
[\tfrac12(\sqrt{\eta_g}-\sqrt{\eta_f}),3(\sqrt{\eta_g}-\sqrt{\eta_f})].
$$
\end{itemize}
 
\noindent\underline{(2)(b):} We consider the following cases. In each of these cases
transversality follows by Lemma~\ref{l:transversal-der}.
\begin{itemize}
\item $f$ is type~(i) and $g$ is type~(ii): we have $f'(y)=1$ and $g'(y)<0$ for all $y\in J$.
\item $f$ is type~(i) and $g$ is type~(iii): we have $f''(y)=0$
and $g''(y)< -1$ for all $y\in J$.
\item $f$ is type~(ii) and $g$ is type~(iii): we have $|f''(y)|\leq 1$
and $g''(y)< -1$ for all $y\in J$. \qedhere
\end{itemize}
\end{proof}
%%%%%%%%%%%%%%%%%%%%%%%%%%%%%%%%%%%%%%%%%%%%%%%%%%%%%%%%%%%%%%%%%%%%%%%%%%%%%%%%

%%%%%%%%%%%%%%%%%%%%%%%%%%%%%%%%%%%%%%%%%%%%%%%%%%%%%%%%%%%%%%%%%%%%%%%%%%%%%%%%
\subsubsection{Rectangular Katz--Tao bound}
\label{s:RKT-proof}

In this subsection we establish a rectangular Katz--Tao bound~\eqref{e:RKT-needed}
for the set $\mathbb F_{\tilde r,m}$ defined in~\eqref{e:F-r-m-def}. We start with the main technical lemma,
which uses the analysis of~\S\ref{s:geometry} together with the properties
of moderate sets studied in~\S\ref{s:moderate}.
%%%%%%%%%%%%%%%%%%%%%%%%%%%%%%%%%%%%%%%%%%%%%%%%%%%%%%%%%%%%%%%%%%%%%%%%%%%%%%%%
\begin{lemma}
\label{l:RKT-Y}
Let $\tilde r\in\widetilde{\mathcal R}$, $x\in I_0$, and $I_1,I_2\subset\mathbb R$ be intervals with $|I_j|=r_j$ and $r_1,r_2\in [r,1]$.
Then
\begin{equation}
  \label{e:RKT-Y}
\big|\big\{
\tilde y\in Y^2\cap \Sigma_{\tilde r}\colon \mathcal A_x(\tilde y)\in I_1\times I_2
\big\}\big|_r
\leq C\log(1/r)\bigg({r_1r_2\over (\tilde r+\sqrt r) r^2}\bigg)^\delta.
\end{equation}
\end{lemma}
%%%%%%%%%%%%%%%%%%%%%%%%%%%%%%%%%%%%%%%%%%%%%%%%%%%%%%%%%%%%%%%%%%%%%%%%%%%%%%%%
\begin{proof}
1. We start with some basic reductions. By~\eqref{e:geometry-fns-A} and \eqref{e:geofun-a}--\eqref{e:geofun-c} we have
$$
\mathcal A_x(I_0^2)\subset [-3r_0,3r_0]\times [2c_{\min}(x),5r_0^2].
$$
Thus we may assume that $I_1\subset [-3r_0,3r_0]$ and $I_2\subset [2c_{\min}(x),5r_0^2]$.
We may restrict our attention to $\{y_2\geq y_1\}$, since $\mathcal A_x(y_2,y_1)=\mathcal A_x(y_1,y_2)$.
Splitting $I_1$ into two intervals if necessary, we may
assume that either $I_1\subset [2a_{\min}(x),3r_0]$
or $I_1\subset [-3r_0,2a_{\min}(x)]$. We will only consider the first case,
with the second one handled similarly by switching the roles of $y_1,y_2$
(so that the vertically moderate sets used below become horizontally moderate).
To summarize, we henceforth may assume that
\begin{equation}
\label{e:I-special}
I_1\subset [2a_{\min}(x),3r_0],\quad
I_2\subset [2c_{\min}(x),5r_0^2]
\end{equation}
and it suffices to show the bound (recalling the definition~\eqref{e:Sigma-r-def})
\begin{equation}
  \label{e:RKT-Y-1}
\begin{gathered}
|\Omega\cap Y^2|_r
\leq C\log(1/r)\bigg({r_1r_2\over (\tilde r+\sqrt r) r^2}\bigg)^\delta
\\
\text{where }
\Omega:=
\begin{cases}
\{\tilde y\in \mathbb R^2\colon \tilde r\leq y_2-y_1\leq 2\tilde r,\ \mathcal A_x(\tilde y)\in I_1\times I_2\},& \text{if }\tilde r >0;
\\
\{ \tilde y\in\mathbb R^2\colon 0\leq y_2-y_1\leq 2\sqrt r,\ \mathcal A_x(\tilde y)\in I_1\times I_2\},&
\text{if }\tilde r=0.
\end{cases}
\end{gathered}
\end{equation}
Note that $\Omega$ is contained in the top quadrant $\mathcal Q_{++}(x)$ defined in~\eqref{e:quadrant++}.

\noindent 2. We show the covering bound
\begin{equation}
  \label{e:Omega-r-est}
|\Omega|_r\leq 2000 {r_1r_2\over (\tilde r+\sqrt r) r^2}.
\end{equation}
We first consider the case $\tilde r>0$; note that then $\tilde r\geq \sqrt r$.
We have by~\eqref{e:distance-verified-0}
$$
\mathcal A_x(\Omega(r))\subset I_1(6r)\times I_2(6r).
$$
On the other hand, by~\eqref{e:A-x-J} we have
$$
|\det \partial_{\tilde y}\mathcal A_x(\tilde y)|\geq \tilde r\quad\text{for all }\tilde y\in \Omega(r).
$$
The map $\mathcal A_x$ is injective on $\Omega(r)$ by~\eqref{e:distance-verified-2}.
Thus the change of variables formula gives
$$
\tilde r|\Omega(r)|\leq |\mathcal A_x(\Omega(r))|\leq (r_1+12r)(r_2+12r)\leq 169r_1r_2
$$
Together with~\eqref{e:volume-nbhd}, this gives~\eqref{e:Omega-r-est}:
$$
(r/2)^2|\Omega|_r\leq |\Omega(r)|\leq {169r_1r_2\over \tilde r}.
$$
Now, consider the case $\tilde r=0$. The map
$$
\mathcal B_x:\tilde y\mapsto (F_{\tilde y}(x),y_2-y_1)
$$
is a local diffeomorphism near $(0,0)$ when $x\in I_0$ since
$$
d\mathcal B_0(0,0)=\begin{pmatrix}
1 & 1 \\
-1 & 1
\end{pmatrix}.
$$
We have
$$
\mathcal B_x(\Omega(r))\subset I_1(6r)\times [-\sqrt r,3\sqrt r].
$$
Thus the change of variables formula and~\eqref{e:volume-nbhd} give~\eqref{e:Omega-r-est}:
$$
(r/2)^2|\Omega|_r\leq |\Omega(r)| \leq |\mathcal B_x(\Omega(r))| \leq 4(r_1+12r)\sqrt r
\leq 52{r_1r_2\over \sqrt r}.
$$

\noindent 3. We now split $\Omega$ into vertically moderate sets
using Lemma~\ref{l:moderate-split-general}. We write $I_1=[a_1,a_1+r_1]$ and $I_2=[a_2,a_2+r_2]$.
Recall from~\eqref{e:I-special} that
$$
2a_{\min}(x)\leq a_1\leq a_1+r_1\leq 3r_0,\quad
2c_{\min}(x)\leq a_2\leq a_2+r_2\leq 5r_0^2.
$$
We use~\eqref{e:geometry-fns-A}
to see that for $\tilde r>0$, the set $\Omega$ consists of points $(y_1,y_2)$ which satisfy the inequalities
$$
\begin{gathered}
\tilde r\leq y_2-y_1\leq 2\tilde r;
\\
a_1\leq a_x(y_1)+a_x(y_2)\leq a_1+r_1;
\\
a_2\leq 2(c_{\min}(x)+b_x(y_1)^2+b_x(y_2)^2)\leq a_2+r_2.
\end{gathered}
$$
When $\tilde r=0$, the set $\Omega$ is described in the same way,
except we replace the first line by $0\leq y_2-y_1\leq 2\sqrt r$.

Let $\eta:=\tfrac12 a_2 - c_{\min}(x)$, $\eta':=\tfrac12(a_2 + r_2) - c_{\min}(x)$
and
$$
\begin{aligned}
J\, &:=\big[y_{\min}(x)-0.8\sqrt\eta,y_{\min}(x)+0.8\sqrt\eta\big],
\\
J'\, &:=\big[y_{\min}(x)-0.8\sqrt{\eta'},y_{\min}(x)+0.8\sqrt{\eta'}\big].
\end{aligned}
$$
Note that $0\leq \eta< \eta'\leq 3r_0^2$.

We use Lemma~\ref{l:quadrant-circle}(i) and the fact that $\Omega\subset\mathcal Q_{++}(x)$ to see that
$\Omega\cap\{y_1\in J\}$ can be written as the set of $(y_1,y_2)$ satisfying
$y_1\in J$ and the following inequalities
(with the correct replacement for the first line when $\tilde r=0$):
\begin{equation}
  \label{e:RKT-Y-ineq2}
\begin{aligned}
y_1+\tilde r &\leq y_2\leq y_1 +2\tilde r;
\\
a_x^{-1}(a_1-a_x(y_1)) &\leq y_2\leq a_x^{-1}(a_1+r_1-a_x(y_1));
\\
b_x^{-1}\Big(\sqrt{\eta - b_x(y_1)^2}\Big) &\leq y_2 \leq
b_x^{-1}\Big( \sqrt{\eta' - b_x(y_1)^2}\Big).
\end{aligned}
\end{equation}
The same lemma gives a similar description of $\Omega\cap \{y_1\in J'\setminus J\}$,
except the first inequality in the last line is removed.
That lemma also shows that $\Omega\subset \{y_1 \in J'\}$. 

The inequalities~\eqref{e:RKT-Y-ineq2} have the form $f(y_1)\leq y_2$ or $y_2\leq g(y_1)$
where $f,g$ are admissible functions (see Definition~\ref{d:admissible}). By Lemma~\ref{l:moderate-split-general}
(applied over $J$ and over each of the two intervals in $J'\setminus J$),
whose assumptions are satisfied by Lemma~\ref{l:lines-circles}, we can write
$$
\Omega\subset\bigcup_{\ell=1}^L \Omega_\ell,\quad
L\leq 18\cdot 10^4\log(1/r)+9\cdot 10^5,\quad
|\Omega_\ell|_r\leq 9|\Omega|_r,
$$
where each $\Omega_\ell$ is a vertically moderate region. By Lemma~\ref{l:moderate-regular},
we have for each $\ell$
$$
|\Omega_\ell\cap Y^2|_r\leq C |\Omega_\ell|_r^\delta\leq C|\Omega|_r^\delta.
$$
It follows that
$$
|\Omega\cap Y^2|_r\leq C\log(1/r)|\Omega|_r^\delta
\leq C\log(1/r)\bigg({r_1r_2\over (\tilde r+\sqrt r)r^2}\bigg)^\delta
$$
where in the last inequality we used~\eqref{e:Omega-r-est}. This finishes the proof of~\eqref{e:RKT-Y-1}.
\end{proof}
%%%%%%%%%%%%%%%%%%%%%%%%%%%%%%%%%%%%%%%%%%%%%%%%%%%%%%%%%%%%%%%%%%%%%%%%%%%%%%%%
We now prove a version of Lemma~\ref{l:RKT-Y} for the projection map defined in~\eqref{e:function-projection}:
%%%%%%%%%%%%%%%%%%%%%%%%%%%%%%%%%%%%%%%%%%%%%%%%%%%%%%%%%%%%%%%%%%%%%%%%%%%%%%%%
\begin{lemma}
  \label{l:RKT-projection}
Let $\tilde r\in\widetilde{\mathcal R}$, $x\in I_0$, and $I_1,I_2\subset\mathbb R$ be intervals with $|I_j|=r_j$ and $r_1,r_2\in [r,1]$. Then
\begin{equation}
  \label{e:RKT-projection}
\#\{\tilde{\mathbf y}\in \mathcal D_r(Y^2\cap \Sigma_{\tilde r})\colon A_x(F_{\tilde{\mathbf y},\tilde r})\in I_1\times I_2\}
\leq
C\log(1/r)\bigg({r_1r_2\over (\tilde r+\sqrt r) r^2}\bigg)^\delta.
\end{equation}
\end{lemma}
%%%%%%%%%%%%%%%%%%%%%%%%%%%%%%%%%%%%%%%%%%%%%%%%%%%%%%%%%%%%%%%%%%%%%%%%%%%%%%%%
\begin{proof}
Denote
$$
\begin{aligned}
\mathbf \Upsilon \, &:=
\{\tilde{\mathbf y}\in \mathcal D_r(Y^2\cap \Sigma_{\tilde r})\colon A_x(F_{\tilde{\mathbf y},\tilde r})\in I_1\times I_2\},
\\
\Upsilon \, & :=\{\tilde y\in Y^2\cap \Sigma_{\tilde r}\colon
\mathcal A_x(\tilde y)\in I_1(r)\times I_2(r)\}.
\end{aligned}
$$
Assume that $\tilde{\mathbf y}\in \mathbf\Upsilon$.
By the construction of $F_{\tilde{\mathbf y},\tilde r}$ in the proof of Lemma~\ref{l:function-projection},
there exists $\tilde y_{\tilde{\mathbf y}}\in \tilde{\mathbf y}\cap Y^2\cap \Sigma_{\tilde r}$
such that $\|F_{\tilde y_{\tilde{\mathbf y}}}-F_{\tilde{\mathbf y},\tilde r}\|_{C^2}\leq r$ and thus
$$
|\mathcal A_x(\tilde y_{\tilde{\mathbf y}})-A_x(F_{\tilde{\mathbf y},\tilde r})|\leq r.
$$
This implies that $\mathcal A_x(\tilde y_{\tilde{\mathbf y}})\in I_1(r)\times I_2(r)$, that is
$\tilde y_{\tilde{\mathbf y}}$ lies in $\Upsilon$.

We have shown that each $r$-square in $\mathbf \Upsilon$ intersects~$\Upsilon$,
that is $\mathbf\Upsilon\subset\mathcal D_r(\Upsilon)$.
It follows from~\eqref{e:D-r-bound} that
$$
\#(\mathbf\Upsilon)\leq 9 |\Upsilon|_r.
$$
It remains to estimate $|\Upsilon|_r$ by Lemma~\ref{l:RKT-Y}.
\end{proof}
%%%%%%%%%%%%%%%%%%%%%%%%%%%%%%%%%%%%%%%%%%%%%%%%%%%%%%%%%%%%%%%%%%%%%%%%%%%%%%%%

One consequence of Lemma~\ref{l:RKT-projection} is an upper bound on the multiplicity defined in~\eqref{e:multiplicity-def}:
%%%%%%%%%%%%%%%%%%%%%%%%%%%%%%%%%%%%%%%%%%%%%%%%%%%%%%%%%%%%%%%%%%%%%%%%%%%%%%%%
\begin{lemma}
\label{l:mul-bound}
Let $\tilde r\in\widetilde{\mathcal R}$ and $F\in \mathbb F_{\tilde r}$. Then
\begin{equation}
  \label{e:mul-bound}
M_{\tilde r}(F)\leq  C\log(1/r)(\tilde r+\sqrt r)^{-\delta}.
\end{equation}
\end{lemma}
%%%%%%%%%%%%%%%%%%%%%%%%%%%%%%%%%%%%%%%%%%%%%%%%%%%%%%%%%%%%%%%%%%%%%%%%%%%%%%%%
\begin{proof}
Fix any $x\in I_0$ and any $r$-intervals $I_1,I_2$ such that $A_x(F)\in I_1\times I_2$.
Then
$$
\begin{aligned}
M_{\tilde r}(F) &\leq
\#\{\tilde{\mathbf y}\in \mathcal D_r(Y^2\cap\Sigma_{\tilde r})\colon
A_x(F_{\tilde{\mathbf y},\tilde r})\in I_1\times I_2\}
\\
&\leq C\log(1/r)(\tilde r+\sqrt r)^{-\delta}
\end{aligned}
$$
where in the last inequality we used Lemma~\ref{l:RKT-projection} with $r_1=r_2=r$.
\end{proof}
%%%%%%%%%%%%%%%%%%%%%%%%%%%%%%%%%%%%%%%%%%%%%%%%%%%%%%%%%%%%%%%%%%%%%%%%%%%%%%%%
Another consequence of Lemma~\ref{l:RKT-projection} is the following rectangular Katz--Tao bound
for the sets defined in~\eqref{e:F-r-m-def}:
%%%%%%%%%%%%%%%%%%%%%%%%%%%%%%%%%%%%%%%%%%%%%%%%%%%%%%%%%%%%%%%%%%%%%%%%%%%%%%%%
\begin{lemma}
  \label{l:RKT-bound}
The estimate~\eqref{e:RKT-needed} holds for $\mathbb F:=\mathbb F_{\tilde r,m}$ with
\begin{equation}
  \label{e:RKT-bound-C}
C^{\RKT}_{\tilde r,m}=Cm^{-1}\log(1/r)(\tilde r+\sqrt r)^{-\delta}.
\end{equation}
\end{lemma}
%%%%%%%%%%%%%%%%%%%%%%%%%%%%%%%%%%%%%%%%%%%%%%%%%%%%%%%%%%%%%%%%%%%%%%%%%%%%%%%%
\begin{proof}
Let $I_1,I_2\subset\mathbb R$ be intervals with $|I_1|=r_1$, $|I_2|=r_2$,
and $r_1,r_2\in [r,1]$. Take $x\in I_0$.
Then the bound~\eqref{e:RKT-needed} is proved as follows:
$$
\begin{aligned}
\#(\mathbb F_{\tilde r,m}\cap A_x^{-1}(I_1\times I_2))
&\leq m^{-1}\#\{\tilde{\mathbf y}\in \mathcal D_r(Y^2\cap\Sigma_{\tilde r})\colon
F_{\tilde{\mathbf y},\tilde r}\in \mathbb F_{\tilde r,m}\cap A_x^{-1}(I_1\times I_2)\}
\\
&\leq Cm^{-1}\log(1/r)\bigg({r_1r_2\over (\tilde r+\sqrt r)r^2}\bigg)^\delta
=C^{\RKT}_{\tilde r,m}\Big({r_1r_2\over r^2}\Big)^\delta.
\end{aligned}
$$
Here in the first inequality we use the definition~\eqref{e:F-r-m-def}
and in the second inequality we use Lemma~\ref{l:RKT-projection}.
\end{proof}
%%%%%%%%%%%%%%%%%%%%%%%%%%%%%%%%%%%%%%%%%%%%%%%%%%%%%%%%%%%%%%%%%%%%%%%%%%%%%%%%

%%%%%%%%%%%%%%%%%%%%%%%%%%%%%%%%%%%%%%%%%%%%%%%%%%%%%%%%%%%%%%%%%%%%%%%%%%%%%%%%
\subsection{End of the proof}

Take $\tilde r\in\widetilde{\mathcal R}$ and dyadic $m\geq 1$. For
$\mathbf p\in \mathcal D_r(X\times [-1,1])$, define (recalling~\eqref{e:incidence-set-def})
\begin{equation}
\label{e:W-def}
W_{\tilde r,m}(\mathbf p):=\#\{F\in\mathbb F_{\tilde r,m}\colon (F,\mathbf p)\in\mathscr I_8\}.
\end{equation}
%%%%%%%%%%%%%%%%%%%%%%%%%%%%%%%%%%%%%%%%%%%%%%%%%%%%%%%%%%%%%%%%%%%%%%%%%%%%%%%%
\begin{lemma}
  \label{l:AE-prepare}
We have
\begin{equation}
\label{e:AE-prepare}
|(X\times Y^4)\cap\mathcal E_r(\Theta)|
\leq Cr^5\log(1/r)^4\max_{\tilde r\in\widetilde{\mathcal R}}\max_m \bigg(m^2\sum_{\mathbf p\in \mathcal D_r(X\times [-1,1])}W_{\tilde r,m}(\mathbf p)^2\bigg).
\end{equation}
\end{lemma}
%%%%%%%%%%%%%%%%%%%%%%%%%%%%%%%%%%%%%%%%%%%%%%%%%%%%%%%%%%%%%%%%%%%%%%%%%%%%%%%%
\begin{proof}
Recall from~\eqref{e:AE-rewritten} that
$$
|(X\times Y^4)\cap \mathcal E_r(\Theta)|
=\big|\{(x,\tilde y,\tilde y')\in X\times Y^4\colon
|F_{\tilde y}(x)-F_{\tilde y'}(x)|\leq r\}\big|.
$$
For each $(x,\tilde y,\tilde y')\in X\times Y^4$
such that $|F_{\tilde y}(x)-F_{\tilde y'}(x)|\leq r$,
there exists
$\mathbf z\in \mathcal D_r([-1,1])$ such that
$F_{\tilde y}(x),F_{\tilde y'}(x)\in \mathbf z(r)$.
Take $\mathbf x\in \mathcal D_r(X)$ such that $x\in\mathbf x$
and denote $\mathbf p:=\mathbf x\times \mathbf z\in\mathcal D_r(X\times [-1,1])$.
Then $(x,F_{\tilde y}(x)),(x,F_{\tilde y'}(x))\in \mathbf p(r)$ and thus,
using the notation~\eqref{e:incidence-set-def}, we have
$(F_{\tilde y},\mathbf p),(F_{\tilde y'},\mathbf p)\in \mathscr I_1$.
It follows that
$$
|(X\times Y^4)\cap \mathcal E_r(\Theta)|
\leq r\sum_{\mathbf p\in\mathcal D_r(X\times [-1,1])}
|\{\tilde y\in Y^2\colon (F_{\tilde y},\mathbf p)\in\mathscr I_1\}|^2.
$$
By~\eqref{e:Sigma-union} we have for each $\mathbf p\in\mathcal D_r(X\times [-1,1])$
$$
|\{\tilde y\in Y^2\colon (F_{\tilde y},\mathbf p)\in\mathscr I_1\}|
\leq \sum_{\tilde r\in\widetilde{\mathcal R}}
|\{\tilde y\in Y^2\cap \Sigma_{\tilde r}\colon (F_{\tilde y},\mathbf p)\in\mathscr I_1\}|.
$$
Let $\tilde y\in Y^2\cap\Sigma_{\tilde r}$ be such that $(F_{\tilde y},\mathbf p)\in\mathscr I_1$
and take $\tilde{\mathbf y}\in\mathcal D_r(Y^2\cap\Sigma_{\tilde r})$ such that
$\tilde y\in\tilde{\mathbf y}$. 
By~\eqref{e:function-projection-close}, we have
$\|F_{\tilde y}-F_{\tilde{\mathbf y},\tilde r}\|_{C^2(I_0)}\leq 7r$ and thus
$(F_{\tilde{\mathbf y},\tilde r},\mathbf p)\in \mathscr I_8$. It follows that
$$
\begin{aligned}
|\{\tilde y\in Y^2\cap \Sigma_{\tilde r}\colon (F_{\tilde y},\mathbf p)\in\mathscr I_1\}|
&\leq
r^2\#\{\tilde{\mathbf y}\in \mathcal D_r(Y^2\cap\Sigma_{\tilde r})\colon (F_{\tilde{\mathbf y},\tilde r},\mathbf p)\in\mathscr I_8\}
\\
&=r^2\sum_{F\in\mathbb F_{\tilde r}\atop (F,\mathbf p)\in\mathscr I_8} M_{\tilde r}(F)
\leq 2r^2 \sum_m m W_{\tilde r,m}(\mathbf p).
\end{aligned}
$$
By Cauchy--Schwarz, we then get
$$
\begin{aligned}
|(X\times Y^4)\cap \mathcal E_r(\Theta)|
&\leq 4r^5\sum_{\mathbf p\in\mathcal D_r(X\times [-1,1])}\bigg(
\sum_{\tilde r\in\widetilde{\mathcal R}}\sum_m mW_{\tilde r,m}(\mathbf p)
\bigg)^2
\\
&\leq
Cr^5\log(1/r)^2\sum_{\mathbf p\in\mathcal D_r(X\times [-1,1])}
\sum_{\tilde r\in\widetilde{\mathcal R}}\sum_m m^2W_{\tilde r,m}(\mathbf p)^2
\\
&\leq Cr^5\log(1/r)^4\max_{\tilde r\in\widetilde{\mathcal R}}\max_m \bigg(m^2\sum_{\mathbf p\in \mathcal D_r(X\times [-1,1])}W_{\tilde r,m}(\mathbf p)^2\bigg)
\end{aligned}
$$
which gives~\eqref{e:AE-prepare}.
\end{proof}
%%%%%%%%%%%%%%%%%%%%%%%%%%%%%%%%%%%%%%%%%%%%%%%%%%%%%%%%%%%%%%%%%%%%%%%%%%%%%%%%
We now split the set of points $\mathcal D_r(X\times [-1,1])$
according to the values of $W_{\tilde r,m}$.
For dyadic $w\geq 1$, define
\begin{equation}
  \label{e:P-r-m-w-def}
\mathbb P_{\tilde r,m,w}:=\{\mathbf p\in\mathcal D_r(X\times [-1,1])\colon w\leq W_{\tilde r,m}(\mathbf p)<2w\}.
\end{equation}
Then Lemma~\ref{l:AE-prepare} implies that
\begin{equation}
  \label{e:AE-prepare-more}
|(X\times Y^4)\cap \mathcal E_r(\Theta)|\leq Cr^5\log(1/r)^5\max_{\tilde r\in\widetilde{\mathcal R}}\max_{m,w}
(mw)^2\#(\mathbb P_{\tilde r,m,w}).
\end{equation}
Using Theorem~\ref{t:OWY} (in the form of Corollary~\ref{c:OWY}) together
with the properties established in~\S\ref{s:AE-ingredients}, we show
%%%%%%%%%%%%%%%%%%%%%%%%%%%%%%%%%%%%%%%%%%%%%%%%%%%%%%%%%%%%%%%%%%%%%%%%%%%%%%%%
\begin{lemma}
  \label{l:final-nail}
We have for all $\tilde r\in\widetilde{\mathcal R}$, $m$, $w$
\begin{equation}
  \label{e:final-nail}
(mw)^2\#(\mathbb P_{\tilde r,m,w})\leq C\min\big(
mw(\tilde r+\sqrt r)^\delta r^{-3\delta},
(mw)^{-1}(\tilde r+\sqrt r)^{-\delta}r^{-4\delta-\epsilon}
\big).
\end{equation}
\end{lemma}
%%%%%%%%%%%%%%%%%%%%%%%%%%%%%%%%%%%%%%%%%%%%%%%%%%%%%%%%%%%%%%%%%%%%%%%%%%%%%%%%
\begin{proof}
1. For $F\in\mathbb F_{\tilde r,m}$, define
$$
\mathbb P_{\tilde r,m,w}(F):=\{\mathbf p\in \mathbb P_{\tilde r,m,w}\colon (F,\mathbf p)\in\mathscr I_8\}.
$$
We first verify the Katz--Tao property for this set. Let $I_1\subset\mathbb R$
be an interval of length $r_1\in [r,1]$. Assume that $\mathbf p=\mathbf x\times \mathbf z\in\mathbb P_{\tilde r,m,w}(F)$ and $\mathbf x\in\mathcal D_r(I_1)$. We have $\mathbf x\in \mathcal D_r(X)$,
so $\mathbf x\in \mathcal D_r(X\cap I_1(r))$,
and similarly to~\eqref{e:Y-discrete-count} we have
$$
\#(\mathcal D_r(X\cap I_1(r)))\leq C\Big({r_1\over r}\Big)^\delta.
$$
Next, the condition
$(F,\mathbf p)\in\mathscr I_8$ implies that $\mathbf z\in \mathcal D_r(F(\mathbf x(8r))(8r))$.
The neighborhood $F(\mathbf x(8r))(8r)$ is contained in an interval of length~$33r$, thus 
$\#(\mathcal D_r(F(\mathbf x(8r))(8r)))\leq 35$.
Thus
\begin{equation}
  \label{e:KT-graph-shading}
\#(\mathbb P_{\tilde r,m,w}(F)\cap \mathcal D_r(I_1\times [-1,1]))
\leq C\Big({r_1\over r}\Big)^\delta.
\end{equation}
In particular, taking $I_1=I_0$, we have
\begin{equation}
  \label{e:P-F-basic-bound}
\#(\mathbb P_{\tilde r,m,w}(F))\leq Cr^{-\delta}.
\end{equation}

\noindent 2. Consider the set of incidences
$$
\mathscr I_{\tilde r,m,w}:=(\mathbb F_{\tilde r,m}\times \mathbb P_{\tilde r,m,w})\cap \mathscr I_8.
$$
From the definition~\eqref{e:P-r-m-w-def}, we see that
\begin{equation}
  \label{e:finail-int-1}
\#(\mathscr I_{\tilde r,m,w})
=\sum_{\mathbf p\in\mathbb P_{\tilde r,m,w}}W_{\tilde r,m}(\mathbf p)
\geq w \#(\mathbb P_{\tilde r,m,w}).
\end{equation}
On the other hand, by~\eqref{e:P-F-basic-bound} and~\eqref{e:F-r-m-bound}, we get
$$
\#(\mathscr I_{\tilde r,m,w})
=\sum_{F\in\mathbb F_{\tilde r,m}}
\#(\mathbb P_{\tilde r,m,w}(F))
\leq Cr^{-\delta}\#(\mathbb F_{\tilde r,m})\leq Cm^{-1}(\tilde r+\sqrt r)^\delta r^{-3\delta}.
$$
Together with~\eqref{e:finail-int-1} this gives the first half of~\eqref{e:final-nail}:
$$
\#(\mathbb P_{\tilde r,m,w})\leq C(mw)^{-1}(\tilde r+\sqrt r)^\delta r^{-3\delta}.
$$

\noindent 3. We now apply Corollary~\ref{c:OWY} with $\epsilon$ replaced by $0.1\epsilon$ and
$$
\mathbb F:=\mathbb F_{\tilde r,m},\quad
\mathbb P:=\mathbb P_{\tilde r,m,w}.
$$
The family $\mathbb F_{\tilde r,m}$ is $r$-separated and $\mathfrak T_\Theta$-transversal
as explained right after its definition~\eqref{e:F-r-m-def}.
By Lemma~\ref{l:RKT-bound}, the rectangular Katz--Tao bound~\eqref{e:RKT-needed}
holds with $C^{\RKT}_{\tilde r,m}$ given by~\eqref{e:RKT-bound-C}. For each $F\in\mathbb F$,
the family $\mathbb P_{\tilde r,m,w}(F)$ is an $(r,\sigma,C)$-KT cube family by~\eqref{e:KT-graph-shading},
since $\delta\leq \frac23$ and thus $\delta\leq \sigma$.
Thus the estimate~\eqref{e:OWY-c} holds:
$$
\#(\mathscr I_{\tilde r,m,w})\leq C(C_{\tilde r,m}^{\RKT})^{\frac23}r^{-\frac23\delta-0.1\epsilon}\#(\mathbb P_{\tilde r,m,w})^{\frac23}\#(\mathbb F_{\tilde r,m})^{\frac13}.
$$
Together with~\eqref{e:finail-int-1} this gives
$$
\begin{aligned}
\#(\mathbb P_{\tilde r,m,w}) &\leq
C w^{-3} (C_{\tilde r,m}^{\RKT})^2 r^{-2\delta-0.3\epsilon}\#(\mathbb F_{\tilde r,m})
\\
&\leq
C (mw)^{-3}(\tilde r+\sqrt r)^{-\delta} r^{-4\delta-\epsilon}
\end{aligned}
$$
where in the last inequality we used~\eqref{e:RKT-bound-C} and~\eqref{e:F-r-m-bound}.
This gives the second half of~\eqref{e:final-nail} and finishes its proof.
\end{proof}
%%%%%%%%%%%%%%%%%%%%%%%%%%%%%%%%%%%%%%%%%%%%%%%%%%%%%%%%%%%%%%%%%%%%%%%%%%%%%%%%
We are ready to finish
%%%%%%%%%%%%%%%%%%%%%%%%%%%%%%%%%%%%%%%%%%%%%%%%%%%%%%%%%%%%%%%%%%%%%%%%%%%%%%%%
\begin{proof}[Proof of Proposition~\ref{l:additive-bound}]
Replacing the minimum on the right-hand side of~\eqref{e:AE-prepare-more} by
the geometric mean, we see that
$$
(mw)^2\#(\mathbb P_{\tilde r,m,w})\leq C r^{-\frac72\delta-\frac12\epsilon}.
$$
Together with~\eqref{e:AE-prepare-more} this gives
$$
|(X\times Y^4)\cap \mathcal E_r(\Theta)|\leq C r^{5-\frac72\delta-\epsilon}
$$
which proves~\eqref{e:additive-bound}.
\end{proof}
%%%%%%%%%%%%%%%%%%%%%%%%%%%%%%%%%%%%%%%%%%%%%%%%%%%%%%%%%%%%%%%%%%%%%%%%%%%%%%%%

%%%%%%%%%%%%%%%%%%%%%%%%%%%%%%%%%%%%%%%%%%%%%%%%%%%%%%%%%%%%%%%%%%%%%%%%%%%%%%%%
% BIBLIOGRAPHY
%%%%%%%%%%%%%%%%%%%%%%%%%%%%%%%%%%%%%%%%%%%%%%%%%%%%%%%%%%%%%%%%%%%%%%%%%%%%%%%%
\bibliographystyle{alpha}
\bibliography{General,QC,Dyatlov,Projection}

\end{document}